\documentclass{article}
\usepackage{geometry}               % See geometry.pdf to learn the layout options. There are lots.
\usepackage{graphicx} % Required for inserting images

\usepackage[sc]{mathpazo}
\usepackage{enumerate}

\usepackage[english]{babel}
\usepackage[dvipsnames]{xcolor}
\usepackage{amsmath}
\usepackage{amssymb}
\usepackage{amsfonts}
\usepackage{amsthm,graphicx}
\usepackage{comment}
\usepackage{hyperref}
\hypersetup{
	linktocpage=true,
	colorlinks=true,
	linkcolor=blue!70!black,
	citecolor=orange!40!red,
	urlcolor=magenta!80!black,
}

\usepackage{mathtools}
\mathtoolsset{showonlyrefs}

\makeatletter

\usepackage[normalem]{ulem}

\numberwithin{equation}{section}
\newtheorem{theorem}{Theorem}[section]
\newtheorem{lemma}[theorem]{Lemma}

\newtheorem{proposition}[theorem]{Proposition}
\theoremstyle{definition}  
\newtheorem{definition}[theorem]{Definition}

\newtheorem{remark}[theorem]{Remark}

\newcommand{\mc}{\mathcal}
\newcommand{\mb}{\mathbb}

\newcommand{\la}{\lambda}
\newcommand{\norm}[1]{\left\lVert#1\right\rVert}
\newcommand{\pd}[2]{\frac{\partial#1}{\partial#2}}

\newcommand{\R}{\mb{R}}

\newcommand{\N}{\mb{N}}

\newcommand{\e}{\varepsilon}
\newcommand{\eps}{\e}

\newcommand{\capa}[2]{\textnormal{\rm{Cap}}_{#1}(#2)}

\DeclareMathOperator{\supp}{supp}
\DeclareMathOperator{\dist}{dist}

\newcommand{\shp}{\Omega}
\newcommand{\bx}{D}

\title{Miminization of the first eigenvalue of
the Dirichlet Laplacian with a small volume obstacle}

\author{Benedetta Noris, Giovanni Siclari and Gianmaria Verzini}

\date{\today}

\begin{document}
\maketitle

\begin{abstract}
We consider the well-known shape optimization problem with spectral cost: minimizing the first 
eigenvalue of the Dirichlet Laplacian among all subdomains $\Omega$ having prescribed volume and  contained in a fixed box 
$D$; equivalently, we look for the best way to remove a compact set (obstacle) $K\subset\overline{D}$ of Lebesgue measure $|K|=\eps$,  
$0<\varepsilon<|D|$, in order to minimize the first Dirichlet eigenvalue of the set $\Omega = 
D \setminus K$. 

In the small volume regime $\eps\to0$, we prove that the optimal obstacles accumulate, in a suitable 
sense, to points of $\partial D$ where $|\nabla \phi_0|$ is minimal, where $\phi_0$ denotes the first  
eigenfunction of the Dirichlet Laplacian on $D$. Moreover, we provide a fairly detailed description of the convergence of the optimal eigenvalues, eigenfunctions and free boundaries. Our results are based 
on sharp estimates of the optimal eigenvalues, in terms of a suitable notion of relative capacity.
\end{abstract}
\noindent
{\footnotesize \textbf{AMS-Subject Classification}}.
{\footnotesize 49Q10, 47A75, 35R35.}

\section{Introduction}\label{sec_introduction}

The study of shape optimization problems involving eigenvalues of differential operators has a long 
history, see for instance the monograph \cite{MR2251558}. Beyond fundamental questions, such as the 
existence and regularity of an optimal shape, a main issue in this framework concerns the description of 
the qualitative geometric properties of such shape.

When the problem enjoys some symmetries, for example when the admissible unknown shape $\shp$ is 
allowed to vary in a symmetric class and the differential operator is symmetric too, a typical question deals with 
the symmetry of the optimal shape $\shp^*$ (``is $\shp^*$ a ball?''). This usually reflects on a 
fairly detailed description of the optimizer (``yes, $\shp^*$ is a ball!''). On the other hand, in 
case of asymmetric constraints, this question has generally a negative answer, and the qualitative 
properties of $\shp^{*}$ (connectedness, convexity, star-shapedness, location, ...) are often very 
diffucult to be described. A prototypical case in this direction is that of an inclusion constraint, 
namely when the admissible shape is forced (among other constraints) to lie in a given box $\bx$. 

When a complete description seems out of reach, a natural way to obtain information is to focus on 
a specific regime for the involved parameters. For instance, if also a volume constraint is present, the 
study of the small volume regime is often very effective in providing a detailed description of 
$\shp^*$ in such regime. Also this subject is classical, as testified for instance by the monograph \cite{MR1711532}.
In this setting, one expects concentration to points, or more generally to subsets of low dimension, 
therefore a natural question concerns the precise location of such limit sets within the box $\bx$. To 
detect it, a major role is played by the availability of a sharp asymptotic expansion of the optimal 
energy value with respect to the small volume parameter. Such expansion can be obtained by some blow-up procedure, after 
the classification of the possible entire limit profiles. We refer to the recent papers  \cite{MR4759535,FMPV24}, 
where this strategy has been successfully adopted to analyse the minimization of a weighted eigenvalue 
arising in population dynamics, both with Dirichlet and with Neumann boundary conditions. It turns out 
that the blow-up approach is quite successful when a differential equation holds 
globally in the box $\bx$, as in the previous examples, while it is more difficult to handle in the case of one-phase free boundary problems, see 
e.g. \cite{MR4066101}, also because of the difficulties in providing a full classification of  
entire solutions to one-phase problems \cite{MR4661533,MR4850026}.

On the other hand, in the context of spectral stability theory under possibly singular domain 
perturbations, very much is known about the asymptotic expansions of differential eigenvalues, 
in case of a perturbation consisting in the removal of a prescribed small hole from a bounded 
domain. Indeed, understanding how eigenvalues are sensitive to small variations in the domain is of 
interest in many applications \cite{C_spectrum_holes}, and nowadays very precise formulas are 
available, in a number of different cases 
\cite{MR3950657,FNO_Dirich_bound,MR4489278,MR4391692,MR4848582,Sic25}. Such sharp formulas, many of 
which are very recent, are based on different ad hoc definitions of \emph{(relative) capacity} and, as 
we already mentioned, they are all conceived for explicitly prescribed small perturbations. 
For what we said, it seems natural to try to exploit such formulas to extract qualitative and 
quantitative information in suitable small volume regimes, in situations where the small perturbation is 
not explicit, but instead it is induced by some optimization procedure. This paper can be seen as 
a first attempt to follow this line, and for this we choose one of the most basic spectral shape 
optimization problems, that is the minimization of the principal eigenvalue of the Dirichlet Laplacian 
in a box, in the presence of an obstacle having small volume.
\bigskip

Let $\bx \subset \R^N$, $N \ge 2$, be a bounded domain (open and connected set) with boundary of class $C^{2,\alpha}$, with $\alpha>0$.
For every $\e\in [0,|\bx|)$, where $|\cdot|$ denotes the $N$-dimensional Lebesgue measure, we consider the minimization problem 
\begin{equation}\label{min_prob_eigen_ge}
\begin{split}
\lambda_\e :&= \inf\{\la_1(\bx\setminus K):K\subset \overline{\bx} \text{ compact, } |K|\ge\e\}\\
&= \inf\{\la_1(\shp):\shp\subset \bx \text{ open, } |\shp|\le|\bx|-\e\},
\end{split}
\end{equation}
where 
\begin{equation}\label{eq_R_quot_approx}
 \la_1(\shp)=\min\left\{\frac{\int_{\bx}|\nabla u|^2 \, dx}{\int_{\bx}u^2 \, dx}: \, u \in H^1_0(\shp)\right\}
\end{equation}
denotes the first eigenvalue of the problem 
\begin{equation}\label{prob_eigen_lap_approx}
\begin{cases}
-\Delta u= \la u, & \text{ in } \shp, \\
u=0, &\text{ on } \partial \shp.
\end{cases}
\end{equation}
With this choice, we are already suggesting we are interested in the regime when the volume $|K|=\eps$ 
is small, since the case $|D|$ small is essentially trivial, see below. Notice that, here and in the following, we endow $H^1_0(\bx)$ with the Dirichlet norm, and we consider $H^1_0(\shp) \subset H^1_0(\bx)$, by trivially extending $H^1_0(\shp)$-functions in $\bx$.

By monotonicity with respect to domain inclusion, when $\e=0$, problem \eqref{min_prob_eigen_ge} is achieved by $\shp=\bx$, thus $\la_0=\la_1(\bx)$, the first eigenvalue of the Laplacian with Dirichlet boundary conditions in 
$\bx$:
\begin{equation}\label{eq_phi}
\begin{cases}
-\Delta u= \la u, & \text{ in } \bx, \\
u=0, &\text{ on } \partial \bx.
\end{cases}
\end{equation}
Being $\bx$ bounded and connected, $\la_0$ is positive and simple, and any associated eigenfunction does not change sign. 
In the following we denote by $\phi_0$ the (unique) eigenfunction associated to $\la_0$ that satisfies
\begin{equation}\label{hp_phi}
\norm{\phi_0}_{L^2(\bx)} =1 \quad \text{ and } \quad  \phi_0 >0 \text{ in } \bx.  
\end{equation}

There is a vast literature concerning problem \eqref{min_prob_eigen_ge}, see for 
example \cite{H_min_open}, \cite[Chapter 3]{MR2251558}, \cite[Example 4.12]{B_survey} or \cite[Problem 1]{HZ_op} for an overview. We provide in the following a 
quick, informal presentation of the main known results which we are concerned with, and we refer 
to Appendix \ref{sec_app_prop_min} for a more formal discussion and for precise references to the previous literature. First of all, it is well known that \eqref{min_prob_eigen_ge} is achieved for every $\e \in (0,|\bx|)$, although the minimizer may not be unique. Moreover, any (open) minimizer of \eqref{min_prob_eigen_ge} is connected, with measure $|\bx|-\e$, so that \eqref{min_prob_eigen_ge} is equivalent to 
\begin{equation}\label{min_prob_eigen_=}
\min\{\la_1(\shp):\shp\subset \bx \text{ open, } |\shp|=|\bx|-\e\}.
\end{equation}

We notice that, in principle, every minimizer is defined up to null capacity 
modifications, all associated to the same eigenfunction, seen as an element of 
$H^1_0(\bx)$.  On the other hand, it is known that such eigenfunctions admit a 
representative which is Lipschitz continuous in $\overline{\bx}$, so that 
their positivity set (is open and) can be exactly identified. In the following, we denote by $\phi_{\e}$ such a 
representative, choosing it positive and normalized in $L^2(\bx)$:
\begin{equation}\label{hp_phi_e}
\norm{\phi_{\e}}_{L^2(\bx)} =1, \quad \text{ and } \quad  \phi_{\e} \ge0 \text{ in } \bx,  
\end{equation}
and by $\shp_\e$ its positivity set, namely
\begin{equation}\label{D_e_def}
\shp_\e:=\{ x\in\bx: \,  \phi_\e(x)>0 \}.
\end{equation}
As we already noticed,  $\shp_\e$ is connected for every 
$\e \in (0,|\bx|)$. Furthermore, $\phi_\eps$ is a weak solution of the 
variational problem
\begin{equation}\label{eq_phi_e}
\int_{\bx} \nabla \phi_\e \cdot \nabla \psi \, dx= \la_\e \int_{\bx} \phi_\e  \psi \, dx \quad \text{ for any } \psi \in H^1_0(\shp_\e)
\end{equation}
and there exists a constant $\Lambda_\eps>0$ such that
\begin{equation}\label{eq:visc_phi_e}
\begin{cases}
-\Delta \phi_\e= \la_\e \phi_\e, \text{ in } \shp_\e,\\
|\nabla \phi_\e| = \sqrt{ \Lambda_\e}, \text{ in } \partial \shp_\e \cap \bx,\\
|\nabla \phi_\e| \ge  \sqrt{ \Lambda_\e}, \text{ in } \partial \shp_\e \cap \partial  \bx,
\end{cases}
\end{equation}
where the boundary conditions hold in viscosity sense. This last fact draws a connection between this problem and the classical one-phase Bernoulli problem 
\cite{V_book_free_boun}, and in turn this implies that the free boundary 
$ \partial \shp_\e \cap \bx$ is regular, up to a singular set of small Hausdorff dimension 
(codimension bigger than 4). Once again, 
we refer to Appendix \ref{sec_app_prop_min} for more details and full bibliographic information.

Let us define, for every $\e \in (0,|\bx|)$, 
\[
K_\e:=\overline{\bx\setminus \shp_\e},
\]
so that
\begin{equation}\label{K_e_def}
K_\e \text{ compact, } |K_\e|=\e, \text{ and }
\lambda_1(\bx\setminus K_\e)=\la_\e=\min\{\la_1(\shp): |\shp|=|\bx|-\e\}.
\end{equation}

If $\e \in (0,|\bx|)$ is close enough to $|\bx|$, then the optimal sets $\shp_\eps$, $K_\e$ are known explicitly.
Indeed,  let $\rho(\bx)$ denote the inradius of $\bx$, that is 
\begin{equation}\label{def_in_rad}
\rho(\bx) := \sup\{r>0:B_r(x) \subset \bx \text{ for some } x \in \bx\},
\end{equation}
where, for any $x \in \R^N$ and $r>0$,  we denote by $B_r(x) := \{ y \in \mathbb{R}^N : |x - y| < r \}$ the open ball of radius $r$ centered at $x$. Then, if
\begin{equation}\label{def_e0}
\e \ge\e_0:= |\bx| - \rho(\bx)^N |B_1(0)|,
\end{equation}
by the Faber-Krahn inequality any ball contained in $\bx$ and of measure $|\bx|-\eps$ is a minimizer. 
As we already mentioned, this suggests that the interesting small volume regime is the one with $\eps$ small.

In the present paper, we discuss some properties of the optimal obstacles $K_\e$ and of the 
corresponding eigenpairs $(\la_\e,\phi_\e)$ in the asymptotic regime $\e \to 0^+$, aiming at
a complete description of the limit, without further restrictions on the box $\bx$. In particular, 
this will be done in strict relation with a suitable notion of relative capacity, see Definition 
\ref{def_capa} ahead.

Concerning the sets, our main result 
states that $K_\eps$ accumulates, up to subsequences, to some points of $\partial \Omega$ where 
$|\nabla \phi_0|$ is minimum. More precisely, it is not difficult to see that, for any vanishing sequence, it is always possible to find a subsequence $\e_n\to0$ as $n\to+\infty$, and a compact set $K_0\subset\overline{\bx}$, such that 
\[
K_{\e_n}\xrightarrow{H} K_0,
\]
in the sense of Hausdorff (see Definition \ref{def:Hausdorff}). In this regard, denoting with $\chi_E$ 
the characteristic function of a measurable set $E$, our main result is 
the following.
\begin{theorem}\label{theo_sets_expansion}
Let $\bx \subset \R^N$ be a bounded domain of class $C^{2,\alpha}$, and let $\phi_0$  and $K_{\e}$ 
be as in \eqref{hp_phi}, \eqref{K_e_def} respectively. 

There exist a compact,  
non-empty set $K_0$ and a positive Radon measure $\mu_0$ such that, up to subsequences,
\[
\begin{split}
K_{\e}&\xrightarrow{H} K_0\subset\partial\bx\qquad\text{(in Hausdorff sense)},\\
\frac{1}{\e}\chi_{K_\e} & \xrightharpoonup{\,*\,} \mu_0\qquad\qquad\quad\text{(in the sense of measure),}
\end{split}
\]
as $\eps\to0^+$.  Moreover, $\mu_0(K_0)=1$ and 
\begin{equation}\label{eq_min_in_K_mu}
	\supp \mu_0 \subset K_0 \subset  \{x\in\partial\bx : |\nabla \phi_0(x)|=\min_{ \partial \bx} |\nabla \phi_0|\}.
\end{equation}
\end{theorem}
\begin{remark}\label{rmk:connected?}
According to the previous theorem, it is clear that, in case
\[
 \{x\in\partial\bx : |\nabla \phi_0(x)|=\min_{ \partial \bx} |\nabla \phi_0|\} = \{x_0\},
\]
for some $x_0\in\partial\bx$, then $K_0= \{x_0\} =\supp \mu_0$, too. On the other hand, for specific 
choices of the box $\bx$, such set may contain more than one point: for instance, in case $\bx$ is an 
ellipse in $\R^2$, it consists of two points (in particular, it is also disconnected). A 
more striking example is provided when taking $\shp_\eps$ as the box (although in 
principle $\partial\shp_\eps$ may not be $C^{2,\alpha}$, see Appendix \ref{sec_app_prop_min}); indeed, 
in view of \eqref{eq:visc_phi_e}, we know that the set $\{x\in\partial\shp_\eps : |\nabla \phi_\eps(x)|
=\min_{ \partial\shp_\eps} |\nabla \phi_\eps|\}$ contains the free boundary $\partial \shp_\eps\cap\bx$, and hence it has positive $(N-1)$-Hausdorff measure.
 
In all these cases $K_0$ may be a strict subset of the set above, although we conjecture this should not 
be the case.
\end{remark}

As a matter of fact, the previous theorem follows by refined estimates about the convergence of the associated eigenvalues and eigenfunctions, as depicted in the following result.
\begin{theorem}\label{theo_eigen_expansion}
Under the same assumptions as in Theorem \ref{theo_sets_expansion}, let $\lambda_\e$ be as in \eqref{min_prob_eigen_ge}, $\shp_{\e} $ as in \eqref{D_e_def}, 
and $\phi_\e$ as in \eqref{hp_phi_e}. Then, as $\e \to 0^+$,
\begin{align}
\label{eq_eigen_Ke_precise}
\la_\e-\la_0&=\e \min_{x \in \partial \bx} |\nabla \phi_0(x)|^2+ o(\e),\\
\label{eq_eigenfunction_Ke_precise}
\int_{\shp_\e}|\nabla(\phi_\e-\phi_0)|^2 \, dx  &=\e \min_{x \in \partial \bx} |\nabla \phi_0(x)|^2+ o(\e), \\
\label{eq_eigenfunction_norm_H10}
 \int_{\bx}|\nabla(\phi_\e-\phi_0)|^2 \, dx &= 2\e \min_{x \in \partial \bx} |\nabla \phi_0(x)|^2+ o(\e).
 \end{align}
Furthermore, for any $\alpha \in (0,1)$ 
\begin{equation}\label{limit_norm_holder}
\lim_{\e \to 0^+}\norm{\phi_\e -\phi_0}_{C^{0,\alpha}(\bx)} =0.
\end{equation}
\end{theorem}
\begin{remark}\label{rmk:further_pb}
Theorems \ref{theo_sets_expansion} and \ref{theo_eigen_expansion} are an example of how sharp results in 
the theory of spectral perturbation allow a good description of associated shape optimization problems, 
when some parameter is small. We believe that the same strategy can be fruitfully applied in different 
situations, e.g. involving the removal of zero-measure, small capacity obstacles \cite{MR3369065}, or 
eigenvalues higher than the first one \cite{MR4412076}.
\end{remark}

The uniform convergence of $\phi_\e$ to the regular limit $\phi_0$ has interesting consequences: indeed, 
it allows to apply the improvement of flatness techniques, developed for one-phase free boundary problems in \cite{DS_flatness,CLS_boundary_improv_flatness}, in any free boundary points and uniformly in $\eps$. In turn, this 
provides the following regularity result.
\begin{theorem}\label{theorem_reg_boundary}
Under the previous assumptions and notations, there exists $\e_1 \in (0,\e_0)$ such that  for any $\e \in (0,\e_1)$ we have that 
	$\partial \shp_\e \text{ is a  } C^{1,\frac{1}{2}}$  smooth hypersurfaces  while $\partial \shp_\e \cap \bx$ is analytic. Furthermore
	\begin{equation}\label{ineq_uniform_norms_C1}
		\sup_{\e \in (0, \e_1)} \norm{\phi_\e}_{C^{1,\frac{1}{2}}(\shp_\e)} <+\infty.
	\end{equation}
\end{theorem}

We remark that the main difference in the previous result, with respect to the classical free boundary regularity theory, lies in the absence of singular cones, in any dimension: for $\eps$ sufficiently small, all free boundary points are regular, or equivalently, have a flat blow-up. As $\epsilon \to 0^+$, this becomes intuitively clear from an energetic point of view, as the sets $K_\epsilon$ converge to a subset of the (regular) boundary of $\bx$. 

Our last result concerns the constant value $\sqrt{\Lambda_\eps}$ of the normal derivative of $\phi_\eps$ on the free boundary. Indeed, thanks to the uniform estimate proved in \eqref{ineq_uniform_norms_C1}, we can compute the exact asymptotic expansion of $\sqrt{\Lambda_\e}$ as 
$\e \to 0^+$.

\begin{theorem}\label{theor_Lae_convergence}
Under the previous assumptions, let  $\Lambda_\eps>0$ be defined as in 
\eqref{eq:visc_phi_e}. Then
\begin{equation}\label{limit_Lae_int_phi_0}
 \lim_{\e\to 0^+} \sqrt{\Lambda_\e}= \min_{x \in \partial \bx}  |\nabla \phi_0(x)|.
\end{equation}
\end{theorem}
\begin{remark}\label{rmk:altra}
Arguing as in the proof of Theorem \ref{theo_eigen_expansion} and exploiting Theorem  
\ref{theor_Lae_convergence}, one can also show that
\begin{equation*}%\label{limit_Lae_int_phi_0}
 \lim_{\e\to 0^+}\frac{1}{\e}  \int_{\partial \shp_\e}\phi_0 \, d \mc{H}^{N-1}= \min_{x \in \partial \bx}  |\nabla \phi_0(x)|.
\end{equation*}
See Remark \ref{rem:final_rem} ahead for more details.
\end{remark}

The paper is structured as follows: in Section \ref{sec_spec_stab} we define the notion of 
relative capacity and we exploit the abstract quantitative spectral stability theory developed in 
\cite{BS_quant_spec_stab} to relate such notion with the problem we investigate; Section 
\ref{sec_estimates_capa} is the core of this paper, where we provide sharp estimates of the 
relative capacity of the optimal sets $K_\eps$, as $\eps\to0$, and prove Theorem 
\ref{theo_eigen_expansion}; Section \ref{sec_non_deg} is mainly devoted to the uniform regularity 
of the free boundary, allowing to give the proof of Theorems \ref{theorem_reg_boundary}, 
\ref{theor_Lae_convergence} and to complete the one of Theorem \ref{theo_sets_expansion}; finally, 
as we already mentioned, in Appendix \ref{sec_app_prop_min} we collect some results concerning problem 
\eqref{min_prob_eigen_ge}, already known in the literature, that we use throughout the paper.

\section{Quantitative spectral stability}\label{sec_spec_stab}

The aim of this section is to provide  a quantitative estimate of the rate of convergence of the  
eigenvalues $\lambda_\e$ defined in \eqref{min_prob_eigen_ge} to the limit value $\lambda_0=\lambda_1(\bx)$, as $\e\to0^+$, in terms of an appropriate notion of (relative) capacity. To this aim, we shall apply the theory developed in \cite{BS_quant_spec_stab,FNO_Dirich_bound}.

Let us first recall the definition of convergence of sets in the sense of Hausdorff, see for example \cite[Definition 4.4.11]{AT_book_topic_metric_spaces}. Let, for any $E \subset \R^N$ and $x \in \R^N$,
\begin{equation}\label{eq:dist_def}
\dist(x,E):=\inf\{|x-y|: y \in E\}.
\end{equation}

\begin{definition}\label{def:Hausdorff}
We say that a sequence of sets $\{E_n\}_{n \in \mathbb{N}}$ with $E_n \subset \R^N$
converges to a set $E \subset \R^N$  in the sense of Hausdorff if 
\begin{equation}
\dist_H(E,E_n):=\max\left\{\sup_{x \in E}\dist(x,E_n),\sup_{y \in E_n} \dist(y,E)\right\} \to 0^+,  \quad \text{ as } n \to \infty.
\end{equation}
In this case we write $E_n\xrightarrow{H} E$, as $n \to \infty$. 
\end{definition}

Hausdorff convergence does not ensure spectral stability,  i.e. convergence of the eigenvalues, unless additional geometric conditions are imposed on the sets, see for example \cite[Section 2.4]{B_survey}. As shown in \cite{BS_quant_spec_stab,FNO_Dirich_bound}, and as we are going to see in Theorem \ref{theo_spectral_stability}, an additional assumption on the limit set involving the following notion of capacity ensures the desired convergence of the eigenvalues.

\begin{definition}\label{def_capa}
Let  $E\subset\overline{\bx}$ compact. We define the \emph{relative Dirichlet capacity} of $E$ in $\bx$ as
\begin{equation}\label{min_Dirc_capa}
\capa{\bx}{E,\phi_0}:=\inf \left\{\int_{\bx}|\nabla u|^2 \, dx: \, u \in H^1_0(\bx), \,  u -\phi_0 \in H^1_0(\bx\setminus E)\right\},
\end{equation}
where $\phi_0$ is as in \eqref{hp_phi} and $u-\phi_0$ is trivially extended to zero in $E$. In particular, 
\[
\capa{\bx}{\partial\bx,\phi_0}=0.
\]
\end{definition}

To proceed with our analysis  we need some properties of the relative Dirichlet capacity.
Notice first that, by classical variational methods,  the minimization problem \eqref{min_Dirc_capa} is achieved by a capacitary potential $V_{E,\phi_0} \in H^1_0(\bx)$ which weakly solves the problem 
\begin{equation}\label{prob_VE}
	\begin{cases}
		-\Delta V_{E,\phi_0}= 0, &\text{ in } \bx\setminus E,\\
		V_{E,\phi_0}= 0, &\text{ on } \partial \bx,\\
		V_{E,\phi_0}= \phi_0, &\text{ in } E,\\   
	\end{cases}
\end{equation}
in the sense that it coincides with $\phi_0$ in $E$ and 
$V_{E,\phi_0} -\phi_0 \in H^1_0(\bx\setminus E)$, with 
\begin{equation}\label{eq_VE}
	\int_{\bx} \nabla V_{E,\phi_0} \cdot \nabla \psi \, dx =0, \quad \text{ for any } \psi \in H^1_0(\bx\setminus E).
\end{equation}

We need the following variant of \cite[Proposition 3.3]{FNO_Dirich_bound} (see also \cite[Proposition 2.3]{BC_Dirc_capa}), which is well known in case $\capa{\bx}{E,\phi_0}$ is replaced with the  standard   Sobolev capacity.

\begin{proposition}\label{prop_cap_H1}
	Let  $E\subset\overline{\bx}$ be compact. We have that $\capa{\bx}{E,\phi_0}=0$ if and only if 
	$H^1_0(\bx\setminus E)=H^1_0(\bx)$
\end{proposition}
\begin{proof}
	If $H^1_0(\bx\setminus E)=H^1_0(\bx)$ then the capacitary potential $V_{E,\phi_0} \in H^1_0(\bx\setminus E)$. Hence we may choose $\psi=V_{E,\phi_0}$ in \eqref{eq_VE} and conclude that $V_{E,\phi_0} \equiv 0$ in $\bx$, thanks to the Poincar\'e inequality.  It follows that $\capa{\bx}{E,\phi_0}=0$.  
	
	Conversely, let us show that the assumption $\capa{\bx}{E,\phi_0}=0$ implies $H^1_0(\bx)\subset H^1_0(\bx\setminus E)$. To this aim we notice that, by definition, the assumption $\capa{\bx}{E,\phi_0}=0$ provides the existence of a sequence $\{u_n\}_{n \in \mathbb{N}}\subset H^1_0(\bx)$ 
	such that $u_n-\phi_0 \in H^1_0(\bx\setminus E)$ for every $n$ and $\norm{u_n}_{H^1_0(\bx)} \to 0^+$ as $n\to \infty$.
	Let $\psi \in C^\infty_c(\bx)$ and let $\psi_n:=\frac{(\phi_0-u_n)\psi}{\phi_0}$. Notice that $\psi_n$ is well defined, as $\phi_0>0$ in $\bx$. Moreover, there exists $c>0$ such that
	\begin{equation}
		\inf_{x\in\supp(\psi)} \phi_0(x)=c.
	\end{equation}
	Then there exists $c'>0$ such that 
	\begin{equation}\label{phi0-positive}
		\|\psi-\psi_n \|^2_{H^1_0(\bx)} 
		\leq c' \| (\psi-\psi_n)\phi_0 \|^2_{H^1_0(\bx)}.
	\end{equation}
	Moreover we have
	\begin{multline}
		\| (\psi-\psi_n)\phi_0 \|^2_{H^1_0(\bx)}
		= \norm{u_n\psi}^2_{H^1_0(\bx)}\le 2(1+\la^{-1}_0) 
		\int_{\bx} \left( u_n^2|\nabla \psi|^2+|\nabla u_n|^2\psi^2 \right) \, dx \\
		\le 2(1+\la^{-1}_0)\left(\norm{\nabla \psi}_{L^\infty(\bx)}^2
		+\norm{\psi}_{L^\infty(\bx)}^2\right) \norm{u_n}_{H^1_0(\bx)} \to 0^+,
		\quad \text{ as } n \to \infty.
	\end{multline}
	The last estimate, together with \eqref{phi0-positive}, shows that $C^\infty_c(\bx)\subset H^1_0(\bx\setminus E)$. Then, by density, 
	$H^1_0(\bx) \subset H^1_0(\bx\setminus E)$. As the inverse inclusion is trivially satisfied, we deduce that $H^1_0(\bx) = H^1_0(\bx\setminus E)$.
\end{proof}

\begin{proposition}\label{prop_En_limits}
Let $\{E_n\}_{n \in \mb{N}}$ be a sequence of compact subsets of $\overline \bx$ such that
\begin{equation}
E_n \xrightarrow{H} E \quad \text{ and } \quad \capa{\bx}{E,\phi_0}=0.
\end{equation} 
Then 
\begin{align}
& \lim_{n\to\infty}\capa{\bx}{E_n,\phi_0} = 0, \label{limit_capa}\\
& V_{E_n,\phi_0}\to 0 \text{ strongly in } H^1_0(\bx), \quad \text{ as } n \to \infty.\label{limit_potential}
\end{align}
\end{proposition}
\begin{proof}
The proof follows that of \cite[Proposition 3.8]{FNO_Dirich_bound}. To simplify the notation 
let $V_n:=V_{E_n,\phi_0}$. For any $n \in \mathbb{N}$ we have that $V_n\in H^1_0(\bx)$, $V_n-\phi_0\in H^1_0(\bx\setminus  E_n)$ and
\begin{equation}\label{eq:V_n}
\int_{\bx} \nabla V_n \cdot \nabla \psi \,dx =0  \quad \text{ for any } \psi \in H^1_0(\bx \setminus E_n) .
\end{equation}
As $\phi_0$ is an admissible test function in \eqref{min_Dirc_capa}, we infer that 
$\|V_n\|_{H^1_0(\bx)} \leq \|\phi_0\|_{H^1_0(\bx)}$
for any $n \in \mathbb{N}$.
Hence $\{V_n\}_{n\in \N}$ is bounded in $H^1_0(\bx)$ and there exist $V \in H^1_0(\bx)$ and a subsequence, still denoted by $\{V_{n}\}_{n \in \mb{N}}$, such that $V_n\rightharpoonup V$ weakly in $H^1_0(\bx)$ as $n\to\infty$, that is
\begin{equation}\label{eq:V_n_k}
\int_\bx \nabla V_n \cdot \nabla \psi \, dx \to 
\int_\bx \nabla V\cdot\nabla \psi \,dx
\quad \text{ for any } \psi \in H^1_0(\bx),
\end{equation}
as $n\to\infty$.  Let $\psi \in C^\infty_c(\bx\setminus E)$. 
As $E_n \xrightarrow{H} E$,  it eventually holds $\psi \in C^\infty_c(\bx\setminus E_n)$, so that both \eqref{eq:V_n} and \eqref{eq:V_n_k} hold, thus leading to 
\begin{equation}\label{eq:V=0}
\int_{\bx} \nabla V \cdot \nabla \psi \,dx =0    \quad \text{ for any } \psi \in C^\infty_c(\bx\setminus E),
\end{equation}
and hence, by density, for every $\psi \in H^1_0(\bx\setminus E)$.
Then, Proposition  \ref{prop_cap_H1} ensures that \eqref{eq:V=0} holds for every $\psi \in H^1_0(\bx)$, so that $V=V_{E,\phi_0}=0$. Furthermore, taking $\psi=V-\phi_0$ in \eqref{eq:V=0} and $\psi=V_n-\phi_0$ in \eqref{eq:V_n_k}, we obtain
\begin{multline*}
0=\capa{\bx}{K_0,\phi_0}=\int_{\bx} |\nabla V|^2 \, dx=\int_{\bx} \nabla V \cdot\nabla \phi_0\, dx\\=\lim_{n\to \infty }\int_{\bx} \nabla V_n\cdot \nabla \phi_0 \, dx 
=\lim_{n \to \infty }\capa{\bx}{E_n,\phi_0}.
\end{multline*}
Thanks to  the Urysohn Subsequence Principle we conclude that \eqref{limit_capa} and \eqref{limit_potential} hold.
\end{proof}

In view of Propositions \ref{prop_cap_H1} and \ref{prop_En_limits}, together with \cite[Theorem  3.8]{BS_quant_spec_stab}, we are now in a  position to apply \cite[Theorem 3.1,  Proposition 3.4]{BS_quant_spec_stab}, thus obtaining the following result.
\begin{theorem}[{\cite[Theorem 3.1, Proposition 3.4]{BS_quant_spec_stab}}]\label{theo_spectral_stability}
Let $\bx \subset \R^N$, $N \ge 2$, be a bounded Lipschitz domain. Let $\phi_0$ be as in \eqref{hp_phi}.
Let $\{E_n\}_{n \in \mb{N}}$ be a sequence of compact subsets of $\overline \bx$ such that
\begin{equation}\label{hp_En_E}
   E_n \xrightarrow{H} E \quad \text{ and } \quad \capa{\bx}{E,\phi_0}=0,
\end{equation}
with $E\subset \overline \bx$ compact.
Then $\lim_{n\to\infty}\capa{\bx}{E_n,\phi_0}=0$ and, denoting by $\{\lambda_h(\bx)\}_{h\geq1}$, the sequence of eigenvalues of \eqref{eq_phi},  we have
\begin{equation}\label{eq_eigenvalues_asym}
\la_{h}(\bx\setminus E_n)=\la_{h}(\bx)+\capa{\bx}{E_n,\phi_0}+o(\capa{\bx}{E_n,\phi_0}), \quad  \text{ as } n \to \infty,
\end{equation}
for every $h\geq1$.
Furthermore, denoting by $\phi_{E_n}$ the positive, $L^2(\bx)$-normalized eigenfunction associated to $\la_1(\bx\setminus E_n)$, it holds
\begin{align}
&\norm{\nabla \phi_{E_n}-\nabla \phi_0}^2_{L^2(\bx\setminus E_n)}=\capa{\bx}{E_n,\phi_0}+o(\capa{\bx}{E_n,\phi_0}), \quad  \text{ as } n \to \infty, \label{conv_eigenfunctions_H1_norm}\\
&\norm{\phi_{E_n}-\phi_0}_{L^2(\bx\setminus E_n)}^2=o(\capa{\bx}{E_n,\phi_0})), \quad  \text{ as } n \to \infty. \label{conv_eigenfunctions_L2_norm}
\end{align}
\end{theorem}

Turning now to the main minimization problem \eqref{min_prob_eigen_=} we are interested in, 
let $\{K_{\e} \}_{\e \in (0,|\bx|)}$ be a family of optimal obstacles, i.e. a family of sets satisfying \eqref{K_e_def}.
It follows (see e.g. \cite[Theorem 4.4.15 and Remark 4.4.16]{AT_book_topic_metric_spaces}) that there exist a compact,  non-empty set $K_0 \subset \overline {\bx}$ and a sequence $\{K_{\e_n} \}_{n \in \mathbb{N}}$ such that
\begin{equation}\label{H-conv}
K_{\e_n} \xrightarrow{H} K_0
\quad \text{and}\quad \e_n \to0,
\qquad \text{as } n\to\infty.
\end{equation}
In general, the set $K_0$ depends on the choice of the subsequence $\{K_{\e_n} \}$.
We shall prove that any limit set $K_0$ has zero relative Dirichlet capacity.
As a consequence,  Theorem \ref{theo_spectral_stability} applies, thus providing the convergence up to subsequences of $\lambda_\e$ to $\lambda_0$. More precisely, in the present section we prove the following result.
\begin{theorem}\label{theo_la1_Ke_easymptotic}
Let $\{K_{\e} \}_{\e \in (0,|\bx|)}$ be a family of optimal sets as in \eqref{K_e_def} and let $\phi_{\e}$ be as in \eqref{eq_phi_e}, \eqref{hp_phi_e}.
There exist a sequence $\{K_{\e_n} \}_{n \in \mathbb{N}}$ and a compact,  non-empty set $K_0 \subset \overline {\bx}$ such that $K_{\e_n} \xrightarrow{H} K_0$ and $\e_n\to0$ as $n\to\infty$. Moreover, for any such sequence, the following holds
\begin{align}
&\lim_{n \to \infty}\capa{\bx}{K_{\e_n},\phi_0}=\capa{\bx}{K_0,\phi_0}=0, 
\qquad \text{ and  } \quad K_0\subset\partial \bx; \label{eq_capa_meas_K}\\
&  \lambda_{\e_n} 
=\lambda_0 
+\capa{\bx}{K_{\e_n},\phi_0}+o(\capa{\bx}{K_{\e_n},\phi_0}),
\end{align}
Furthermore,  as $n \to \infty$, it holds
\begin{align}
	&\norm{\nabla \phi_{\e_n}-\nabla \phi_0}^2_{L^2(\bx\setminus K_{\e_n})}=\capa{\bx}{K_{\e_n},\phi_0}+o(\capa{\bx}{K_{\e_n},\phi_0}), \label{conv_eigenfunctions_H1_norm_K}\\
	&\norm{\phi_{\e_n}-\phi_0}_{L^2(\bx\setminus K_{\e_n})}^2=o(\capa{\bx}{K_{\e_n},\phi_0})). \label{conv_eigenfunctions_L2_norm_K}
\end{align}
\end{theorem}

In the remaining part of this section, we shall prove this result. 
Let us first show that the family $\{\lambda_\e\}$ converges to $\lambda_0$ as $\e\to0^+$.

\begin{lemma}\label{lem_imit_eigen_Ke}
We have that
\begin{equation}\label{limit_eigen_Ke}
\lim_{\e \to 0^+} \la_\e=\la_0.
\end{equation}
\end{lemma}
\begin{proof}
Let $x_0 \in \bx$ and $r_\e=(\e/|B_1|)^{\frac{1}{N}}$, with $\e$ sufficiently small that $B_{r_\e}(x_0)\subset \bx$. Then, since $\capa{\bx}{\{x_0\}}=0$ and $ B_{r_\e}(x_0) \xrightarrow{H} \{x_0\}$ as $\e \to 0^+$,  by  Theorem \ref{theo_spectral_stability} we have
\begin{equation}
\la_1(\bx\setminus \overline{B_{r_\e}(x_0)})=\la_0+\capa{\bx}{\overline{B_{r_\e}(x_0)},\phi_0}+o(\capa{\bx}{\overline{B_{r_\e}(x_0)},\phi_0}), \quad  \text{ as } \e \to 0^+.
\end{equation}
In particular $\la_1(\bx\setminus \overline{B_{r_\e}(x_0)}) \to \la_0$ as $\e \to 0^+$.

As $|B_{r_\e}(x_0)|=\e$, the set $\bx\setminus\overline{B_{r_\e}(x_0)}$ is admissible for the minimization problem \eqref{min_prob_eigen_ge}, hence
\begin{equation}
\limsup_{\e \to 0^+}   \la_\e \leq
\lim_{\e \to 0^+}  \la_1(\bx\setminus \overline{B_{r_\e}(x_0)})
=\la_0.
\end{equation}
On the other hand, by inclusion, 
\begin{equation}\label{ineq_eigen}
 \la_1(\bx\setminus E) \ge  \la_0 \quad \text{ for any compact set } E  \subset \overline{\bx},
\end{equation}
so that, in particular, $\la_\e \ge  \la_0$ for any $\e\in (0,|\bx|)$.
\end{proof}

As a matter of fact, as a consequence of Lemma \ref{lem_imit_eigen_Ke} one can see that the 
eigenfunctions $\phi_\eps$ are equi-Lipschitz continuous for $\eps$ small, see Proposition 
\ref{prop_equi_lip} ahead.

Quite naturally, as the sets $K_\e$ Haudorff converge to $K_0$, and, by the previous lemma, 
$\lambda_\e=\lambda_1(\shp_\e)\to\lambda_0$, letting 
\begin{equation}\label{eq:Dzero_def}
\shp_0:=\bx\setminus K_0,  
\end{equation}
we also have that $\la_1(\shp_0)=\la_0$, as we show in the following lemma.

\begin{lemma}\label{lem_eq_eigen_K}
Let $K_{\e_n}$ and $K_0$ be as in \eqref{H-conv} and $\phi_{\e_n}$ be as in \eqref{eq_phi_e}, \eqref{hp_phi_e}. 
Finally, let $\shp_0$ be as in \eqref{eq:Dzero_def}.
Then
\begin{equation}
\la_1(\shp_0)=\la_0 \qquad\text{and}\qquad
\phi_{\e_n}\to\phi_0
\text{ in } H^1_0(\bx)\cap C^{0,\alpha}(\overline{\bx}) 
\end{equation}
for every $0<\alpha<1$.
\end{lemma}
\begin{proof}
Thanks to Lemma \ref{lem_imit_eigen_Ke} we have 
\[
\|\phi_{\e_n}\|_{H^1_0(\bx)}^2 = \lambda_{\e_n} \to \lambda_0 \quad\text{as } n\to\infty.
\]
Then, up to passing to a subsequence, there exists a function $\phi \in H^1_0(\bx)$ such that $\phi_{\e_n} \rightharpoonup \phi$ weakly in $H^1_0(\bx)$.
Thanks to the Rellich-Kondrachov Theorem, $\norm{\phi}_{L^2(\bx)}=1$ and thus $\phi \ge 0$ in $\bx$ is not trivial.

Let $\psi \in C^\infty_c(\shp_0)$. As $K_{\e_n} \xrightarrow{H} K_0$,  we have that $\psi \in C^\infty_c(\shp_{\e_n})$  eventually,  so that \eqref{eq_phi_e} holds with $\e=\e_n$ for sufficiently large values of $n$. Then, by the weak $H^1$-convergence of $\phi_{\e_n}$ and by Lemma \ref{lem_imit_eigen_Ke}, we have
\begin{equation}
\int_{\bx} \nabla \phi \cdot \nabla \psi \, dx =\la_0  \int_{\bx} \phi  \psi \, dx,
\end{equation}
meaning that $\la_0 $ is the first eigenvalue of problem  \eqref{prob_eigen_lap_approx} with $\shp=\shp_0$, thus leading to the first part of the statement. In turn, this 
implies $\phi=\phi_0$ and, by convergence of the $H^1_0$-norms and the equi-Lipschitz property in Proposition \ref{prop_equi_lip}, we obtain also the second part of the statement.
\end{proof}

As the final ingredient in the proof of Theorem \ref{theo_la1_Ke_easymptotic}, we show that $K_0$ is contained in the boundary of $\bx$.

\begin{lemma}\label{lem_cap_measure}
Let $K_0$ be as in \eqref{H-conv}. Then 
\[
K_0 \subset \partial \bx.
\] 
In particular, this implies $\capa{\bx}{K_0,\phi_0}=0$ and $|K_0|=0$.
\end{lemma}
\begin{proof}
Pick any $x_0\in \bx$. Then, by uniform convergence (see Lemma \ref{lem_eq_eigen_K}), there exists $r>0$ such that 
$B_r(x_0)\subset \bx$ and
\[
\inf_{B_r(x_0)} \phi_{\eps_n} \ge \frac{1}{2} \phi_0(x_0) >0,
\]
for $n$ sufficiently large. As a consequence, $B_r(x_0)\subset \shp_{\eps_n}$ eventually, and $x_0\not\in K_0$. Since $x_0\in\bx$ is arbitrary, the lemma follows.
\end{proof}

\begin{proof}[End of the proof of Theorem \ref{theo_la1_Ke_easymptotic}]
As we already mentioned, the existence of a compact,  non-empty $K_0 \subset \overline {\bx}$ and of a sequence $\{K_{\e_n} \}_{n \in \mathbb{N}}$ such that $K_{\e_n} \xrightarrow{H} K_0$ and $\e_n \to0$ as $n\to\infty$ follows from \cite[Theorem 4.4.15, Remark 4.4.16]{AT_book_topic_metric_spaces}. 
The properties stated in \eqref{eq_capa_meas_K} follow from  Proposition \ref{prop_En_limits}  and Lemma \ref{lem_cap_measure}.
As $\capa{\bx}{K_0,\phi_0}=0$,  the remaining results in Theorem \ref{theo_la1_Ke_easymptotic} descend from Theorem  \ref{theo_spectral_stability}.
\end{proof}

\section{Asymptotics of  \texorpdfstring{$\capa{\bx}{K_\e,\phi_0}$}{the relative capacity}}\label{sec_estimates_capa}

The aim of the present section is to detect the asymptotic behaviour of $\capa{\bx}{K_\e,\phi_0}$ as $\e \to 0^+$. As a consequence, combining with the asymptotic expansion stated in Theorem \ref{theo_la1_Ke_easymptotic},  we prove Theorem  \ref{theo_eigen_expansion} and, in the next section, Theorem \ref{theo_sets_expansion}. 
As we are  interested in the regime $\e \to 0^+$, in the following we suppose that $\e <\e_0$, with $\e_0$ as in \eqref{def_e0}.

We first observe that, since $\bx$ is of class $C^{2,\alpha}$, the classical Hopf's Lemma yields
\begin{equation}\label{ineq_min_nabla_phi0}
\min_{x \in \partial \bx} |\nabla\phi_0(x)|>0.
\end{equation}
Furthermore, being  $\phi_0$ Lipschitz, 
\begin{equation}\label{ineq_phi_boundary}
\phi_0(x) \le L \, \dist(x,\partial \bx) \quad \text{ for any } x \in \bx,
\end{equation}
with $L$ positive.

As a first step, we prove that the relative Dirichlet capacity of $K_\e$ in $\bx$ has vanishing order at most $\e$.
\begin{proposition}\label{prop_lower_estimates_capa}
We have that 
\begin{equation}\label{ineq_lower_estimates_capa}
  \liminf_{\e \to 0^+}\frac{\capa{\bx}{K_\e,\phi_0}}{\e}\ge  \min_{x \in \partial \bx} |\nabla\phi_0(x)|^2.
\end{equation}
Furthermore, if $K_{\e_n} \xrightarrow{H} K_0$ as 
$\e_n \to0$, then 
\begin{equation}\label{ineq_lower_estimates_capa_K}
 \liminf_{n \to \infty}\frac{\capa{\bx}{K_{\e_n},\phi_0}}{\e_n}\ge  \min_{x \in K_0} |\nabla\phi_0(x)|^2=\min_{x \in \partial \bx} |\nabla\phi_0(x)|^2. 
\end{equation}
\end{proposition}

\begin{proof}
We first prove \eqref{ineq_lower_estimates_capa_K}. By \eqref{def_capa} we have
\begin{equation}\label{proof_lower_estimates_capa:1}
\frac{\capa{\bx}{K_\e,\phi_0}}{\e} = \frac{\int_{\bx}|\nabla V_{K_\e}|^2\,dx}{\e} \ge  \int_{\bx}|\nabla \phi_0|^2\frac{\chi_{K_{\e}}}{\e} \,dx,
\end{equation}
where 
\begin{equation}\label{def_chi_Ke}
\chi_{K_{\e}}(x):=
\begin{cases}
1, & \text{ if } x \in K_{\e},\\
0, & \text{ if } x \not \in K_{\e}.
\end{cases}
\end{equation}

Let $\{K_{\e_n}\}$ and $K_0$ as in the statement.
Letting $\lambda_{N}$ denote the $N$-dimensional Lebesgue measure and $\mu_n:=\frac{\chi_{K_{\e_n}}}{\e_n} \lambda_{N}$, it follows that  $\{\mu_n\}_{n \in \mathbb{N}}$ is a sequence of outer Radon measures in $\R^N$ such that 
\begin{equation}
\sup_{n \in \mathbb{N}}\mu_n(H) <+\infty \quad \text{ for any compact subset } H \subset \R^N.
\end{equation}
By \cite[Theorem 1.4.1]{EG_book}, there exists an outer Radon measure $\mu$ and an unrelabelled 
subsequence  $\{\mu_{n}\}_n$ such that $\mu_{n} \stackrel{\ast}{\rightharpoonup} \mu$  weakly in the sense of measures, that is 
\begin{equation}
\int_{\R^N} f \, d\mu_{n}\to \int_{\R^N} f \, d\mu, \quad \text{ for any } f \in C_c^0(\R^N).
\end{equation} 

We claim that $\mu$ is concentrated on $K_0$, that is $\mu(A)=\mu(A \cap K_0)$ for any Borel set $A \subset \R^N$.
Indeed, let $A \subset \R^N$  be  compact with $A \cap K_0=\emptyset$. Let $U$ be an open neighbourhood  of  $K_0$ such that $U \cap A=\emptyset$.  Since   $K_{\e_{n}}\xrightarrow{H} K_0$ as $n \to \infty$, there exists $n_0 \in \mb{N}$ such that $K_{\e_{n}}\subset U$ for any $n\ge n_0$. Consider a positive function $f \in C_c^0(\R^N)$ such that $f\equiv 1$ on $A$ and $f \equiv 0$ on $U$. 
Then 
\begin{equation}
\mu(A) \le 	\int_{\R^N} f \, d\mu= \lim_{n \to \infty} \int_{\R^N} f \, d\mu_{n}=0,
\end{equation}
and the claim follows for any compact $A$; by \cite[Lemma 1.1]{EG_book} the same result extends to any  Borel set, thus proving the claim.

As a consequence, if  $f \in C_c^0(\R^N)$ and $f\equiv 1$ in $\overline{\bx}$,
\begin{equation}
\mu(K_0)=\int_{K_0} f d\mu =\int_{\overline{\bx}} f d\mu= \lim_{n \to \infty} \int_{\overline{\bx}} f\, d\mu_{n}=\lim_{n \to \infty} \int_{\overline{\bx}} \, d\mu_{n}=1.
\end{equation}
Hence,
\begin{equation}
  \lim_{n \to \infty}\int_{\bx}|\nabla \phi_0|^2\frac{\chi_{K_{\e_{n}}}}{\e_{n}} \,dx
=\int_{K_0}|\nabla \phi_0|^2\, d\mu \ge \min_{x \in K_0}|\nabla \phi_0(x)|^2 ,
\end{equation}
which, combined with \eqref{proof_lower_estimates_capa:1} and by virtue of the Urysohn Subsequence Principle, concludes the proof of \eqref{ineq_lower_estimates_capa_K}.

Finally notice that, for any Hausdorff limit $K_0$, we have
\[
\min_{x \in K_0}|\nabla \phi_0(x)|^2 
\ge \min_{x \in \partial \bx}|\nabla \phi_0(x)|^2.
\]
Hence  \eqref{ineq_lower_estimates_capa} holds for the whole family $K_\e$.
\end{proof}

To obtain an upper estimate on $\capa{\bx}{K_\e,\phi_0}$, we shall define a local diffeomorphism that straightens the boundary of $\bx$ and use it to build suitable competitors for problem \eqref{min_prob_eigen_=}, see equation \eqref{def_He} ahead.
To this end, we fix a point $x_0 \in \partial \bx$, and, to simplify the notation, we assume that $x_0=0$. 
Since $\bx$ is a $C^{2,\alpha}$ domain, there exist $r_0>0$  and a function $g\in C^{2,\alpha}(\R^{N-1})$ such that 
\begin{equation}\label{def_g}
	B_{r_0}\cap \bx =\{x \in B_{r_0}: x_N > g(x') \} \quad\text{ and } \quad B_{r_0}\cap \partial \bx =\{x \in B_{r_0}: x_N = g(x') \},
\end{equation}
where $x=(x_1,\dots,x_N)=(x',x_N)$, for any $x \in \R^{N-1} \times \R$, and $B_{r_0}=B_{r_0}(0)$.
Furthermore, up to choosing an appropriate coordinate system, we may assume that
\begin{equation}\label{eq_g}
	g(0)=0 \quad \text{ and } \quad \nabla g(0)=0.
\end{equation}
In particular the tangent hyperplane to $\partial \bx$ in $0$ is $\{x\in \R^N:x_N=0\}$ while the outer normal vector to $\bx$ in  $0$ is 
$\nu:=(0,\dots,0,-1)$. 

Now, we straighten the boundary of $\bx$ near $0$ by means of a suitable local 
diffeomorphism. A possible choice is the following one, see e.g. \cite[Section 2, Appendix A]{NT_sh_Neu} 
or \cite[Section 3]{MPV_op_weighted} for further details on the calculations.
Let us define 
\begin{equation*}
	F:B_{r_0} \to \R^N,  \quad    F(y)=(F_1(y),\dots,F_N(y))
\end{equation*}
as
\begin{equation}\label{def_F}
	F_j(y):=
	\begin{cases}
		\displaystyle y_j-y_N\pd{g}{x_j}(y'), \text{ for  } j=1,\dots N-1,\smallskip\\
		y_N+g(y'), \text{ for  } j=N,
	\end{cases}
\end{equation}
with $g$ as in \eqref{def_g}, \eqref{eq_g}. Then $F\in C^{1,\alpha}(B_{r_0},\R^{N}) $  and, by \eqref{eq_g}, we have that $D F(0)={\rm{Id}}$. Hence the function $F$ is locally invertible. 
It follows that there exists $0<r_1\le r_0$ such that 
\begin{equation}\label{eq_F_diff}
	F:B_{r_1} \to F(B_{r_1}) \text{ is a diffeomorphism, }\quad  F(B^+_{r_1})\subset \bx \quad 
	\text{ and } \quad F(B'_{r_1})=\partial\bx \cap B_{r_1},
\end{equation}
where for any $r>0$
\begin{equation}\label{def_Br}
B_r^+:=\{x \in B_r: x_N>0\}
\quad\text{and}\quad
B'_r:=\{x=(x',0) \in \R^N: |x'|<r\}.
\end{equation}
Let $G:=F^{-1}$, so that
\begin{equation}\label{eq_G_diff}
G: F(B_{r_1}^+) \to B_{r_1}^+
\quad\text{and}\quad
G \in C^{1,\alpha}(F(B^+_{r_1})).
\end{equation}
By \cite[Lemma A.1]{NT_sh_Neu} we have that 
\begin{align}
	&\det{DF(0)}=1-\Delta g(0)\, y_N+O(|y|^2)=1+O(|y|),\quad  \text{ as }  y \to 0, \notag  \\
	&DG(F(y))=(DF(y))^{-1}={\rm{Id}}+O(|y|),\quad \text{ as } y \to 0, \notag \\
	&\det DG(F(y))=1+O(|y|),\quad \text{ as } y \to 0. \label{eq_det_G}
\end{align}
Finally, let $\tilde{F}$ be the restriction of $F$ to $B_{r_1}'$, that is
\begin{equation}\label{def_tilde_F}
	\tilde{F}:B_{r_1}' \to \partial \bx, \quad   
	\tilde{F}(y):=(y',g(y')) \quad  \text{ for any } y \in B_{r_1}'.
\end{equation}

We now build the competitors for problem \eqref{min_prob_eigen_=}, starting from the subgraph of a 
(nontrivial) function $h \in C^{1,\alpha}(\R^{N-1},[0,1])$, such that $\operatorname{supp}h=\overline{B'_{1/2}}$.
Denote by $V$ the volume of the region bounded by the graph of $h$ and the plane $\{y_N=0\}$, that is
\[
V:=\int_{B_1'} h(y') \, dy' =
\left|\{y=(y',y_N): \, (y',0)\in \overline{B_1'}, \, 0\leq y_N\leq h(y')\}\right|.
\]
For every  $\eta \ge 0$ and $r >0$,  we define
\begin{equation}\label{def_R_eta_r0}
R_{r,\eta}:=\left\{y=(y',y_N): (y',0) \in \overline{B'_r}, \,  0\leq y_N \le \frac{\eta}{r^{N-1}V}  h( y'/ r)\right\},
\end{equation}
so that
\begin{equation*}
R_{r,\eta} \subset \overline{B_r^+} \quad \text{ and } \quad |R_{r,\eta}|=\eta,
\end{equation*}
for every $r\in(0,r_1)$ and $0\le \eta<\min\left\{r,\frac{r^NV\sqrt{3}}{2}\right\}=:2\hat\eta$.
\begin{lemma}\label{lem_exist_etae}
Let $F$ and $r_1$ be as in \eqref{def_F} and \eqref{eq_F_diff}.
Let $r\in(0,r_1)$ and $\hat\eta$ be as above.
The function $f_r:[0,2\hat\eta) \to [0,+\infty)$ defined as
\begin{equation}
f_r(\eta):=|F(R_{r,\eta})|
\end{equation}
is continuous and increasing. 
In particular, $f(\left[0,\hat\eta\right])=\left[0,|F(R_{r,\hat\eta})|\right]$. 
\end{lemma}
%and 
\begin{proof}
If $\eta_1<\eta_2$ then $R_{r,\eta_1}\subset  R_{r,\eta_2}$ and therefore $f$ is increasing. Furthermore,  since $F \in C^1(\overline{B_{r_1}^+},\R^{N})$,
\begin{multline}
f(\eta_2)-f(\eta_1)=|F(R_{r,\eta_2}\setminus R_{r,\eta_1})| \\
\le \norm{\det DF}_{L^{\infty}(B_r^+)} |R_{r,\eta_2}\setminus R_{r,\eta_1}|
=\norm{\det DF}_{L^{\infty}(B_r^+)}(\eta_2 -\eta_1),
\end{multline}
for any $\eta_1,\eta_2 \in \left[0,\frac{r_1}{2}\right]$. Hence $f$ is continuous. 
\end{proof}

In the following we let $r\in (0,r_1)$. Let $\e_1=\e_1(r):=\min\{\e_0,|F(R_{r,\hat\eta})|\}$, with $\e_0$ as in \eqref{def_e0}. 
Thanks to Lemma \ref{lem_exist_etae},  for any $\e\in [0,\e_1]$ there exists a unique
\begin{equation*}
\eta_\e \in \left[0,\hat\eta\right] \quad \text{ such that } \quad |F(R_{r,\eta_\e})|=\e
\end{equation*}
(in particular, $\eta_0=0$). Let us define 
\begin{equation}\label{def_He}
	H_{r,\e}:= F(R_{r,\eta_\e}).
\end{equation}
Then $H_{r,\e}$ is  compact subset of $\overline{\bx}$ and  it is a competitor for problem \eqref{min_prob_eigen_=} for any $ r \in (0,r_1)$.
Notice that
\begin{equation}\label{eq:Hr0}
H_{r,0}\equiv F(\overline{B'_r})\subset  \partial \bx, \qquad \capa{\bx}{H_{r,0},\phi_0}=0
\end{equation}
and that
\begin{equation}\label{eq:Hreps}
H_{r,\e} \xrightarrow{H} H_{r,0} \quad \text{as }\e \to 0^+.
\end{equation}

\begin{lemma}\label{lem_ineq_eta_e}
For any $\e\in (0,\e_1)$, it holds
\begin{align}
&\norm{\det DF}_{L^{\infty}(B_r^+)}^{-1} \e \le \eta_\e \le 	 \norm{\det DG}_{L^{\infty}(F(B_r^+))}\e,\label{ineq_eta_e}\\
&\,\mathcal{H}^{N-1}(H_{r,0})\le \norm{\sqrt{\det([D\tilde{F}]^tD\tilde{F})}}_{L^{\infty}(B'_r)}\mc{H}^{N-1}(B'_r).\label{ineq_eta_0}
\end{align}   
\end{lemma}
\begin{proof}
It is enough  to notice that 
\begin{equation}
\e= |F(R_{r,\eta_\e})|\le \norm{\det DF}_{L^{\infty}(B_r^+)}|R_{r,\eta_\e}|= \norm{\det DF}_{L^{\infty}(B_r^+)}  \eta_\e
\end{equation}
and similarly
\begin{equation}
\eta_\e= |G(F(R_{r,\eta_\e}))|\le \norm{\det DG}_{L^{\infty}(F(B_r^+))}|F(R_{r,\eta_\e})|=\norm{\det DG}_{L^{\infty}(F(B_r^+))}   \e.
\end{equation}
Finally, since $H_{r,0}= \tilde{F}(\overline{B'_r})$,
\begin{equation}
	\mathcal{H}^{N-1}(H_{r,0}) = \int_{\{|y'|<r\}}\sqrt{\det([D\tilde{F}]^tD\tilde{F})} \, dy',
\end{equation}
hence \eqref{ineq_eta_0} is proved.
\end{proof}

\begin{lemma}\label{lem_capa_He}
The limit 
\begin{equation}\label{limit_capa_He}
\lim_{\e \to 0^+}\frac{\capa{\bx}{H_{r,\e},\phi_0}}{\e}
\end{equation}
exists, and is finite and strictly positive.
\end{lemma}
\begin{proof}
For any $\e \in (0,\e_1)$, let 
\[
\mathcal{U}_\e(R_{r,\eta_\e}):=\{y : \dist(y,R_{r,\eta_\e}) < \e\}\cap \{y :y_N>0\}.
\]
Since we are interested in the limit $\e \to 0^+$, it is not restrictive to suppose that 
\[
\mathcal{U}_\e(R_{r,\eta_\e}) \subset B^+_{r_1}.
\]
Moreover, since
\[
\mathcal{U}_\e(R_{r,\eta_\e})\subset \left\{(y',y_N) : |y'| < \frac{r}{2} + \e,\ 0< y_N < \frac{\eta_\eps}{r^{N-1}V} + \eps \right\},
\]
\eqref{ineq_eta_e} yields the existence of a universal $C>0$ such that
\begin{equation}\label{ineq_eta_e+}
|\mathcal{U}_\e(R_{r,\eta_\e})| \leq C \e.
\end{equation}

Let $\rho_\e  \in C_c^\infty(\R^N)$ be a cut-off function such that 
\begin{equation}\label{def_rho_e}
|\nabla \rho_\e| \le \e^{-1} \quad \text{ and } \quad 
\rho_\e=
\begin{cases}
1, &\text{ in } R_{r,\eta_\e},\\
0, &\text{ in } \R^N \setminus \mathcal{U}_\e(R_{r,\eta_\e}).
\end{cases}
\end{equation}
Let us consider the function $\phi_0\cdot (\rho_\e \circ G)$. By extending $\rho_\e \circ G$ to $0$ outside of $F(B_{r_1}^+)$, we have that $\phi_0\cdot (\rho_\e \circ G) \in H^1_0(\bx)$.
Then,  by a change of variables,
\begin{multline}\label{eq:phi_0_rho_e_circG1}
\capa{\bx}{H_{r,\e},\phi_0} \le  \int_{F(B^+_{r_1})} |\nabla (\phi_0 \cdot (\rho_\e \circ G))|^2 dx\\
\le 2\int_{F(B^+_{r_1})} |\nabla \phi_0|^2 | \rho_\e \circ G|^2 dx +2\int_{F(B^+_{r_1})} |\phi_0|^2 |\nabla ( \rho_\e \circ G)|^2 dx \\
\le C \left[\int_{B^+_{r_1}} |\nabla \phi_0( F(y))|^2 | \rho_\e|^2 dy +\int_{B^+_{r_1}} |\nabla\rho_\e(y)|^2 |\phi_0(F(y))|^2  dy\right],
\end{multline}
where $ C>0$ only depends on $\norm{G}_{C^1(F(B_{r_1}))}$.
By \eqref{ineq_eta_e+} we have
\begin{equation}\label{eq:phi_0_rho_e_circG2}
\int_{B^+_{r_1}} |\nabla \phi_0(F(y))|^2 | \rho_\e |^2 \, dy \le \norm{\nabla \phi_0}^2_{L^\infty(\bx)} | \mathcal{U}_\e(R_{r,\eta_\e})|
\le C \e,
\end{equation}
for some positive constant $C$.
Furthermore we can find a positive constant $C$,  that depends only on $DF$, such that 
$y \in  \mathcal{U}_\e(R_{r,\eta_\e})$ implies
\begin{equation}
\begin{split}
\dist(F(y),\partial \bx) &\le C |y_n| \le C (\eta_\e +\e )\\
&\le C (\norm{\det DG}_{L^{\infty}(B_r^+)}+1) \e,
\end{split}
\end{equation}
thanks to \eqref{ineq_eta_e}.
Then, in view of  \eqref{ineq_phi_boundary}, \eqref{ineq_eta_e+} and \eqref{def_rho_e},
\begin{equation}\label{eq:phi_0_rho_e_circG3}
\int_{B^+_{r_1}} | \nabla \rho_\e(y)|^2 |\phi_0(F(y))|^2 \, dy  =
\int_{\mathcal{U}_\e(R_{r,\eta_\e})} | \nabla \rho_\e(y)|^2 |\phi_0(F(y))|^2 \, dy  \le  C  \e,
\end{equation}
for some positive constant $C$ independent of $\e$.
In conclusion, by combining \eqref{eq:phi_0_rho_e_circG1}, \eqref{eq:phi_0_rho_e_circG2} and 
\eqref{eq:phi_0_rho_e_circG3}, we have shown that $\capa{\bx}{H_{r,\e},\phi_0}=O(\e)$ as $\e \to 0^+$. Since  $H_{r,\e_1} \subset H_ {r,\e_2}$ whenever $\e_1<\e_2$, the limit \eqref{limit_capa_He} exists. Arguing exactly as in Proposition \ref{prop_lower_estimates_capa}, we conclude that it is also strictly positive.
\end{proof}

For $r\in(0,r_1)$ and $\e \in (0,\e_1(r))$, let us define
\begin{equation}\label {def_O_Sigma_e}
O_{r,\e}:= \bx \setminus H_{r,\e} \quad \text{ and }\quad  \Sigma_{r,\e}:= \partial  O_{r,\e} \cap \bx.
\end{equation}
It is worth noticing that  $O_{r,\e}$ is a connected $C^{1,\alpha}$ domain, in particular we may integrate by parts on $O_{r,\e}$. As usual, let 
\[
\la_1(O_{r,\e}) = \lambda_1(\bx\setminus H_{r,\e})
\] 
denote the first eigenvalue of problem \eqref{prob_eigen_lap_approx} in the set $O_{r,\e}$, and let $\phi_{H_{r,\e}}$ be the associated  positive eigenfunction,  normalized so that $\norm{\phi_{H_{r,\e}}}_{L^2(\bx)}=1$.

\begin{remark}\label{remark_phi_He}
In view of \eqref{def_He},  the boundary of the domain $O_{r,\e}$ is uniformly  $C^{1,\alpha}$ with respect to $\e$, that is, there is an atlas of local charts  bounded in  $C^{1,\alpha}$-norm with respect to  $\e$. In particular by standard elliptic regularity theory,
\begin{equation}\label{ineq_phi_He_bounded}
\sup\{ \norm{\phi_{H_{r,\e}}}_{C^{1,\alpha}(O_{r,\e})}:\e\in [0,\e_1]\} < + \infty.
\end{equation}
Furthermore, if $\nu_{H_{r,\e}}$ is the outer normal vector  to $O_{r,\e}$ on $ \partial O_{r,\e}$ then 
\begin{equation*}
\pd{\phi_{H_{r,\e}}}{\nu_{H_{r,\e}}}(x)\le 0 \quad  \text{ for any } x \in \partial O_{r,\e},
\end{equation*}
since $\phi_{H_{r,\e}}$ is positive. Moreover, for any $\psi \in  H^1(\bx)$, we have
\begin{equation}\label{eq_phi_He_by_part}	
 \int_{O_{r,\e}} (\nabla\phi_{H_{r,\e}} \cdot \nabla \psi -\la_1(O_{r,\e}) \phi_{H_{r,\e}} \psi)\, dx = \int_{\partial O_{r,\e}}\pd{\phi_{H_{r,\e}}}{\nu_{H_{r,\e}}} \psi \, d \mc{H}^{N-1}.
\end{equation} 
\end{remark}

Under the above notation, we can estimate $\la_1(O_{r,\e})$ in terms of $\capa{\bx}{H_{r,\e},\phi_0}$.
\begin{lemma}\label{lem_laH_la_0}
We have that, as $\e \to 0^+$,
\begin{equation}\label{eq_laH_la_0}
\la_1(O_{r,\e})-\la_0=   \capa{\bx}{H_{r,\e},\phi_0} +o(\e)=-\int_{\Sigma_{r,\e}} \phi_0 \pd{\phi_{H_{r,\e}}}{\nu_{H_{r,\e}}} 
\, d \mc{H}^{N-1} + o(\e^{3/2}).
\end{equation}
\begin{equation}\label{eq:phiHeps-phi0}
\norm{\phi_{H_{r,\e}}-\phi_0}_{L^2(O_{r,\e})}^2=o(\e).
\end{equation}
\end{lemma}
\begin{proof}
By \eqref{eq:Hr0} and \eqref{eq:Hreps},  we can apply Theorem  \ref{theo_spectral_stability} (to any 
sequence $\eps_n\to0^+$). In combination with Lemma \ref{lem_capa_He}, this yields both 
\eqref{eq:phiHeps-phi0} and the first equality in \eqref{eq_laH_la_0}, which in turn implies
\[
\la_1(O_{r,\e})-\la_0 = O(\eps).
\]
To prove the second equality in \eqref{eq_laH_la_0} notice that,  by choosing $\psi=\phi_0$ in \eqref{eq_phi_He_by_part}, we obtain
\begin{equation}\label{eq:lambdaO-lambda0}
- \int_{\Sigma_{r,\e}} \phi_0 \pd{\phi_{H_{r,\e}}}{\nu_{H_{r,\e}}} \, d \mc{H}^{N-1}=	(\la_1(O_{r,\e})-\la_0) \int_{O_{r,\e}} \phi_0 \phi_{H_{r,\e}} \, dx= O(\eps).
\end{equation}
By \eqref{conv_eigenfunctions_L2_norm}, Lemma \ref{lem_capa_He} and the H\"older inequality, we can estimate the integral in the right hand side of \eqref{eq:lambdaO-lambda0} as 
\begin{equation}
\int_{O_{r,\e}} \phi_0 \phi_{H_{r,\e}} \, dx= 1-\int_{O_{r,\e}} \phi_0 (\phi_{H_{r,\e}}-\phi_0 )\, dx=1 +o(\e^{1/2}).
\end{equation}
Hence
\begin{equation}
\la_1(O_{r,\e})-\la_0 = -  (1 +o(\e^{1/2}))\int_{\Sigma_{r,\e}} \phi_0 \pd{\phi_{H_{r,\e}}}{\nu_{H_{r,\e}}} \, d \mc{H}^{N-1}
= -\int_{\Sigma_{r,\e}} \phi_0 \pd{\phi_{H_{r,\e}}}{\nu_{H_{r,\e}}} \, d \mc{H}^{N-1}+o(\e^{3/2}),  
\end{equation}
and the proof is complete.
\end{proof}

To proceed,  we exploit elliptic regularity theory to improve the convergence of $\phi_{H_{r,\e}}$ 
to $\phi_0$, at least faraway from $H_{r,0}$.
\begin{lemma}\label{lemma_conv_reg}
Let $\delta>0$ be sufficiently small, and let us define
\[
\bx_\delta := \{x \in \bx: \dist(x,H_{r,0}) >\delta\}.
\]
Then
\[
\lim_{\eps\to0^+} \norm{\phi_{H_{r,\e}}-\phi_0}_{C^{1,\alpha}(\bx_\delta)}=0.
\]
As a consequence, 
\[
\pd{\phi_{H_{r,\e}}}{\nu}(x) \to \pd{\phi_0}{\nu}(x)
\qquad\text{for every }x\in\partial\bx\setminus H_{r,0}
\]
as $\eps\to0^+$.
\end{lemma}
\begin{proof}
Let  $\delta>0$ be fixed and $w_\eps:=\phi_{H_{r,\e}}-\phi_0$. By construction, if $\eps$ is 
sufficiently small we have that $\bx_{\frac{\delta}{2}}\subset O_{r,\eps}$, whence
$w_\eps$ is a solution of the problem 
\begin{equation}
\begin{cases}
-\Delta w_\eps - \la_1(O_{r,\e}) w_\eps = (\la_1(O_{r,\e})-\la_0)\phi_0, &\text{ in } \bx_{\frac{\delta}{2}}, \\
w_\eps=0, &\text{ on }  \partial \bx_{\frac{\delta}{2}}\cap \partial \bx.
\end{cases}
\end{equation}
We are in a position to apply \cite[Corollary 8.36]{GT_book} to $w_\eps$, obtaining that
\begin{equation}
\norm{w_\eps}_{C^{1,\alpha}(\bx_\delta)} \le C\left( \norm{w_\eps}_{C^{0,\alpha}(\bx_{
\frac{\delta}{2}})}+ (\la_1(O_{r,\e})-\la_0) \right),
\end{equation}
for some positive constant $C$ depending on $\delta$ but not on $\eps$ (as long as it is sufficiently small). Then the first conclusion follows by Lemma \ref{lem_laH_la_0} and Remark \ref{remark_phi_He}, and the second one by the fact that for every $x\in\partial\bx\setminus H_{r,0}$ there exists $\delta>0$ such that 
$x\in\partial\bx_{\delta}\cap \partial \bx$.
\end{proof}

Thanks to Lemma  \ref{lemma_conv_reg}, we are in position to prove the following.
\begin{lemma}\label{lem_limit_Sigmae}
We have that 
\begin{equation}\label{eq_limit_Sigmae}
\lim_{\e \to 0^+} \int_{\Sigma_{r,\e}}\pd{\phi_{H_{r,\e}}}{\nu_{H_{r,\e}}} \, d \mc{H}^{N-1}=\int_{H_{r,0}}\pd{\phi_0}{\nu}  \, d \mc{H}^{N-1},
\end{equation}
where $\nu$ is the outer normal vector to $\bx$ on $H_{r,0}\subset \partial  \bx$.
\end{lemma}
\begin{proof}
Choosing $\psi=1$ in \eqref{eq_phi_He_by_part}, we obtain
\begin{equation}\label{proof:prop_limit_Sigmae:1}
-\la_1(O_{r,\e})	\int_{O_{r,\e}}  \phi_{H_{r,\e}} \, dx = \int_{\Sigma_{r,\e}}\pd{\phi_{H_{r,\e}}}{\nu_{H_{r,\e}}}\, d \mc{H}^{N-1}+\int_{\partial \bx \cap \partial O_{r,\e}}\pd{\phi_{H_{r,\e}}}{\nu} \, d \mc{H}^{N-1}.
\end{equation} 
In view of Lemma \ref{lem_laH_la_0} we have
\begin{equation}
\lim_{\e \to 0^+}\la_1(O_{r,\e})	\int_{O_{r,\e}}  \phi_{H_{r,\e}} \, dx=\lim_{\e \to 0^+}\la_1(O_{r,\e})	\int_{\bx}  \phi_{H_{r,\e}} \, dx = \la_0 \int_{\bx}  \phi_0 \, dx.
\end{equation}
On the other hand, by Lemma  \ref{lemma_conv_reg} we already know that $\pd{\phi_{H_{r,\e}}}{\nu} \to \pd{\phi_0}{\nu}$ pointwise in $\partial \bx \setminus H_{r,0} $. Hence, by Remark \ref{remark_phi_He} and Lebesgue's Dominated Convergence Theorem,
\begin{equation}
\lim_{\e  \to \infty}\int_{\partial \bx \cap \partial O_{r,\e}}\pd{\phi_{H_{\e_n}}}{\nu} \, d \mc{H}^{N-1} =\int_{\partial \bx \setminus H_{r,0}}\pd{\phi_0}{\nu}  \, d \mc{H}^{N-1}.
\end{equation}
Since  $\phi_0$ solves  \eqref{eq_phi}
\begin{equation}
	\la_0 \int_{\bx}  \phi_0 \, dx=\int_{\partial \bx }\pd{\phi_0}{\nu}  \, d \mc{H}^{N-1},
\end{equation}
and so, passing to the limit in \eqref{proof:prop_limit_Sigmae:1}, we  have proved \eqref{eq_limit_Sigmae}.
\end{proof}
\begin{remark}\label{rmk:competvsminim}
Differently from the previous results, the last lemma heavily relies on the uniform 
$C^{1,\alpha}$-bounds enjoyed by the competitors $\phi_{H_{r,\eps}}$. For this reason, at this point 
we cannot argue in the same way using the minimizers $\phi_{\eps}$, instead. This will be possible
after the proof of Theorem \ref{theorem_reg_boundary} is completed, see Remark \ref{rem:final_rem} 
ahead.
\end{remark}

\begin{lemma}\label{lem_upper_estimates_capa_He}
Let $\tilde F$ be as in \eqref{def_tilde_F}. Then
\begin{equation}\label{ineq_upper_estimates_capa_He}
\begin{split}
\lim_{\e \to 0^+} &\frac{\capa{\bx}{H_{r,\e},\phi_0}}{\e}\\
&\le \norm{\sqrt{\det([D\tilde{F}]^tD\tilde{F})}}_{L^{\infty}(B_r')}\norm{\det DG}_{L^{\infty}(B_r^+)}
\norm{DF}_{L^\infty(B_{r_1}^+)} \max_{x \in H_{r,0} }|\nabla \phi_0(x)|^2.
\end{split}
\end{equation}
\end{lemma}
\begin{proof}
For any $\e\in (0,\e_1)$ let $y_\e \in R_{r,\eta_\e}$ (recall \eqref{def_R_eta_r0}) be such that 
\begin{equation}
	\phi_0(F(y_\e))=\max\{\phi_0(F(y)):y \in \partial R_{r,\eta_\e} \}=\max\{\phi_0(x):x \in \partial H_{r,\e} \}.
\end{equation}
Then, there exist a subsequence $\{y_{\e_n}\}_{n \in \mb{N}}$ and $y_0 \in R_{r,0}$ such that $y_{\e_n} \to y_0$, as $n \to \infty$. Clearly $F(y_0)\in H_{r,0} \subset \partial \bx$.  
Since $\phi_0\circ F \in C^{1,\alpha}(\overline{B_{r_1}^+})$, with a Taylor expansion in $y$   we can see that
\begin{multline}
\max\{\phi_0(x):x \in \partial H_{r,\e} \}=	\phi_0(F(y_\e))\le |y_0-y_\e| \cdot|\nabla \phi_0(F(y_0))| \cdot\norm{DF}_{L^\infty(\overline{B_{r_1}^+})} +o(|y_0-y_\e|) \\
\le \frac{\e}{\mc{H}(B'_{\frac{r}{2}})} \norm{\det DG}_{L^{\infty}(B_r^+)}
 \norm{DF}_{L^\infty(\overline{B_{r_1}^+})} \max_{x \in H_{r,0} }|\nabla \phi_0(x)| + o(\e),
\end{multline}
in view of \eqref{ineq_eta_e} and \eqref{def_R_eta_r0}. 
Then, taking into account Remark \ref{remark_phi_He}, \eqref{eq_limit_Sigmae} and \eqref{ineq_eta_0},
\begin{multline}
\lim_{\e \to 0^+} \e^{-1}\left|\int_{\Sigma_{r,\e}}\pd{\phi_{H_{r,\e}}}{\nu_{H_{r,\e}}} \phi_0 \, d \mc{H}^{N-1}\right| 
=-\lim_{\e \to 0^+} \e^{-1}\int_{\Sigma_{r,\e}}\pd{\phi_{H_{r,\e}}}{\nu_{H_{r,\e}}} \phi_0 \, d \mc{H}^{N-1} \\
\le \frac{1}{\mc{H}^{N-1}(B'_{\frac{r_1}{2}})} \norm{\det DG}_{L^{\infty}(B_r^+)}
\norm{DF}_{L^\infty(\overline{B_{r_1}^+})} \max_{x \in H_{r,0} }|\nabla \phi_0(x)| \int_{H_{r,0}}\pd{\phi_0}{\nu} \, d \mc{H}^{N-1}\\
\le \norm{\sqrt{\det([D\tilde{F}]^tD\tilde{F})}}_{L^{\infty}(B_r')}\norm{\det DG}_{L^{\infty}(B_r^+)}
\norm{DF}_{L^\infty(\overline{B_{r_1}^+})} \max_{x \in H_{r,0} }|\nabla \phi_0(x)|^2.
\end{multline}
Hence we can deduce \eqref{ineq_upper_estimates_capa_He} from \eqref{eq_laH_la_0}.
\end{proof}

\begin{proof}[Proof of Theorem \ref{theo_eigen_expansion}]
Let $\{K_{\e} \}_{\e \in (0,|\bx|)}$ be a family of optimal sets as in \eqref{K_e_def} and let 
$x_0 \in \partial \bx$ be a global minimum for   $|\nabla \phi_0(x)|$. Up to a translation, let us 
assume that $x_0=0$ and, consequently, let $H_{r,\e}$ be defined as in \eqref{def_He} (centered at 
$x_0=0$), for every $r\in (0,r_1)$, $\eps\in (0,\eps_1(r))$.

Since $H_{r,\e}$ are competitors for problem \eqref{min_prob_eigen_=}, by minimality of $K_{\e}$ we have, for every $r\in (0,r_1)$,
\begin{equation}
\lim_{\e \to 0^+}\frac{\capa{\bx}{K_\e,\phi_0}}{\e} \le \lim_{\e \to 0^+}\frac{\capa{\bx}{H_{r,\e},\phi_0}}{\e}.
\end{equation}
Passing to the limit as  $r\to 0^+$, we conclude that, in view of \eqref{def_F},  \eqref{eq_det_G}, \eqref{def_tilde_F} and \eqref{ineq_upper_estimates_capa_He},
\begin{equation}
	\limsup_{\e \to 0^+}\frac{\capa{\bx}{K_\e,\phi_0}}{\e} \le|\nabla \phi_0(x_0)|^2=  \min_{x \in\partial \bx  }|\nabla \phi_0(x)|^2.
\end{equation}
Then from \eqref{ineq_lower_estimates_capa}  we obtain  \begin{equation}\label{limit_capa_Ke_precise}
\lim_{\e \to 0^+}\frac{\capa{\bx}{K_\e,\phi_0}}{\e}=\min_{x \in \partial \bx} |\nabla \phi_0(x)|^2.
\end{equation}
Hence   \eqref{eq_eigen_Ke_precise} follows from Theorem \ref{theo_la1_Ke_easymptotic}.
Furthermore, we can deduce \eqref{eq_eigenfunction_Ke_precise} from Theorem 
\ref{theo_la1_Ke_easymptotic} and \eqref{eq_eigen_Ke_precise}. Arguing as in 
Proposition \ref{prop_lower_estimates_capa},  \eqref{eq_eigenfunction_norm_H10} follows 
from \eqref{eq_eigenfunction_Ke_precise}. Finally \eqref{limit_norm_holder} was proved 
in Lemma \ref{lem_eq_eigen_K}.
\end{proof}

\section{Uniform regularity of the free boundary}\label{sec_non_deg}

In this section we conclude the proof of our main results. To this aim, we first obtain a lower bound 
for $\sqrt{\Lambda_\e}$ (recall \eqref{eq:visc_phi_e}) as $\e \to 0^+$, so that we can rule out degeneracy.
The proof is an application of Hopf's Lemma. For easier notation, for any $\delta>0$ we define 
\begin{equation}\label{def_A_delta}
A_\delta:=\{x \in \bx: \dist(x,\partial \bx) \ge\delta\}.
\end{equation}

\begin{lemma}\label{lem_non_dege}
There exists a constant $\eta>0$, depending only on $N$ and $\phi_0$, such that
\begin{equation}\label{ineq_estimate_Lae_below}
\liminf_{\e \to 0^+}\sqrt{\Lambda_\e} \ge\eta.
\end{equation}
\end{lemma}
\begin{proof}
We being by noticing that the uniform interior ball condition holds true in $\bx$: 
there exists $R>0$ such that for any  $z \in \partial \bx$, there exists $y_z \in \bx$  such that 
$B_R(y_z) \subset \bx$, and $ z \in  \overline {B_R(y_z)} \cap \partial \bx$. In particular, we 
can  choose $R>0$ uniform with respect to  $z$ since $\partial \bx$ is $C^{2,\alpha}$.

Let $\e_n$, $K_{\e_n}$ and $K_0$ be as in \eqref{H-conv}, i.e. $K_{\e_n} \xrightarrow{H} K_0$ as $\e_n \to0$. We proved in Lemma \ref{lem_cap_measure} 
that $K_0 \subset \partial \bx$. Let $n_0>0$ be such that $\dist_H(K_{\e_n}, K_0) < \frac{R}{2}$ for any $n \ge n_0$. Then, for any such $n$,
\begin{equation}
\partial \shp_{\e_n}  \cap \bx \cap \bigcup_{z \in \mathcal{U}_\delta(K_0) \cap \partial \bx} B_R(y_z) \neq \emptyset,
\end{equation}
where $\mathcal{U}_\delta(K_0):=\{x \in \R^N: \dist(x,K_0)<\delta\}$, for some $\delta>0$ small. In particular there exist 
\[
z_n\in \mathcal{U}_\delta(K_0) \cap \partial \bx
\qquad\text{ and }\qquad 
x_n \in \partial \shp_{\e_n}  \cap \bx,
\]
such that, defining 
\[
y_n = y_{z_n}
\qquad\text{ and }\qquad 
R_n = \sup\left\{r>0: B_r(y_n) \subset \shp_{\eps_n}\right\},  
\]
we obtain 
\[
\frac{R}{2}<R_n <R,\qquad 
B_{R_n}(y_{n}) \subset \shp_{\e_n},
\qquad\text{ and }\qquad
x_{n} \in \overline {B_{R_n}(y_n)} \cap \partial \shp_{\e_n} \cap \bx.
\]
By the proof of Hopf's Lemma (see e.g. \cite[Theorem 3.12]{SV_book_PDEs_action}), and Remark \ref{remark_inner_ball_reg},  we conclude that  there exists a dimensional constant $C_N>0$ such that 
\begin{equation}
\sqrt{\Lambda_{\e_n}} \ge C_N \cdot \min\left\{\phi_{\e_n}(x): x \in \partial{ B_{\frac{R_{n}}{2}}(y_{n})}\right\} \ge  C_N \cdot \min\left\{\phi_\e(x): x \in A_{\frac{R}{2}}\right\},
\end{equation}
where $A_{\frac{R}{2}}$ is as in \eqref{def_A_delta}. Passing to the limit as $n \to \infty$ and exploiting the uniform convergence of $\phi_\eps$ to $\phi_0$, it 
follows that
\begin{equation}
\lim_{n \to \infty}\sqrt{\Lambda_{\e_n}} \ge C_N \cdot \min\{\phi_0(x): x \in A_{\frac{R}{2}}\}.
\end{equation}
By the Uryson Subsequence Principle, we obtain \eqref{ineq_estimate_Lae_below}.
\end{proof}

The  previous lemma allows to rule out degeneracy and to apply the improvement of flatness 
techniques developed in [11, 9], uniformly in $\eps$, thus leading to prove Theorem 
\ref{theorem_reg_boundary}.
\begin{proof}[Proof of Theorem \ref{theorem_reg_boundary}]
Let $\delta>0$ be such that 
\[
|\nabla \phi_0(x)|\ge \frac{1}{2}\min_{x \in \partial \bx}|\nabla \phi_0|
\qquad\text{ for any }x \in \overline \bx \setminus A_\delta, 
\]
where $A_{\delta}$ is as in \eqref{def_A_delta}. 
For any $x \in \overline \bx \setminus A_\delta$ and $\sigma>0$ a Taylor expansion provides $r>0$, depending only on $\sigma$, such that 
\begin{equation}
(y\cdot \nabla \phi_0(x) -\sigma )^+\le \phi_0(y)\le (y\cdot \nabla \phi_0(x) +\sigma)^+
\qquad
\text{ for any $y \in B_r(x)$. }
\end{equation}
Then,  for any $\sigma>0$ and any $x \in \overline \bx \setminus A_\delta$, there exists $\e_1\in (0, \e_0)$ such that for any $\e \in (0,\e_1)$ 
\begin{equation}
	(y\cdot \nabla \phi_0(x) -2\sigma )^+\le \phi_\e(y)\le (y\cdot \nabla \phi_0(x) +2\sigma)^+
	\qquad
\text{ for any $y \in B_r(x)$, }
\end{equation}
thanks to \eqref{limit_norm_holder}. Then, as soon as $K_\e \subset \overline \bx \setminus A_\delta$ and $\sigma$ is smaller than a suitable  dimensional constant, we may apply the improvement of  flatness results 
\cite{DS_flatness,CLS_boundary_improv_flatness} in any point $x$ of the free boundary $\partial \shp_\e \cap \bx$ and  $\partial \shp_\e \cap K_\e \cap \partial \bx$, respectively. Hence, the Lipschitz   blow up is unique and flat in any free boundary point; in particular, the regularity at any fixed small $\e$ follows from Theorem \ref{theo_reg_free}. In view of \eqref{ineq_estimate_Lae_below} and \eqref{ineq_grad_estimate}, we also conclude that \eqref{ineq_uniform_norms_C1} holds. Indeed, we can for example follow the arguments exposed in \cite[Chapter 8]{V_book_free_boun} and use boundary elliptic regularity theory. 
\end{proof}

\begin{remark}
The results in \cite{CLS_boundary_improv_flatness} are actually stated for harmonic functions. However, there is no significant difference in dealing with a $C^{0,\alpha}$-term in the right hand side, see \cite{DS_flatness}.
\end{remark}

\begin{proof}[Proof of Theorems \ref{theo_sets_expansion} and \ref{theor_Lae_convergence}]
Theorem \ref{theo_sets_expansion}, with the exception of the second inclusion in 
\eqref{eq_min_in_K_mu}, is  a consequence of Theorem \ref{theo_la1_Ke_easymptotic} and of 
the proof of Proposition \ref{prop_lower_estimates_capa}. We now prove Theorem 
\ref{theor_Lae_convergence}, obtaining \eqref{eq_min_in_K_mu} as a byproduct.

Suppose that $K_{\e_n} \xrightarrow{H} K_0$ as $n \to \infty$ and let $x_0 \in K_0$ be any point. By a connection argument, there exists a sequence of points $x_{n} \in 
\overline{\partial \shp_{\e_n} \cap \bx}$ such that 
$x_{n}\to x_0$ as $n\to \infty$. Let $\eta >0$ and $y \in B_\eta(x_0)\cap \bx$. Then, as $n \to \infty$,
\begin{multline}
|\sqrt{\Lambda_{\e_n}}- |\nabla \phi_0(x_0)||\le |\nabla \phi_{\e_n}(x_{n}) -\nabla \phi_0(x_0)|\\
\le |\nabla \phi_{\e_n}(x_{n}) -\nabla\phi_{\e_n}(y)|+ |\nabla \phi_{\e_n}(y) -\nabla\phi_0(y)|+ |\nabla \phi_0(y) -\nabla \phi_0(x_0)|\\
\le \sup_{\e \in (0,\e_1)} \norm{\phi_\e}_{C^{1, \frac{1}{2}}(\shp_\e)} |x_{\e_n}-y|^{\frac{1}{2}}+ C |y-x_0|+o(1) \le C'\sqrt{\eta},
\end{multline}
for some constants $C, C'>0$ that do not depend on $n$, thanks to Theorem  \ref{theorem_reg_boundary} 
(applied twice). Hence the limit  $\lim_{{\e}_n \to 0^+}\sqrt{\Lambda_{\e_n}}$ exists and 
\begin{equation}
\lim_{n \to \infty}\sqrt{\Lambda_{\e_n}}= |\nabla \phi_0(x_0)|.
\end{equation}
It follows that $|\nabla \phi_0|$ is constant on $K_0$ and, recalling 
\eqref{ineq_lower_estimates_capa_K},
we conclude that 
\begin{equation}
	\lim_{n \to \infty}\sqrt{\Lambda_{\e_n}}= \min_{x \in \partial \bx} |\nabla \phi_0(x)|.
\end{equation}
and that \eqref{eq_min_in_K_mu} holds. Finally, by the Urysohn subsequence principle, we obtain \eqref{limit_Lae_int_phi_0}. 
\end{proof}

\begin{remark}\label{rem:final_rem}
Thanks to Theorem \ref{theorem_reg_boundary}, we can argue as in Section \ref{sec_estimates_capa} 
with $\phi_\eps$ instead of $\phi_{H_{r,\eps}}$ (see also Remark \ref{rmk:competvsminim}). 
As a consequence, one can show that 
\begin{equation}
	\lim_{\e \to 0^+}  \int_{\partial \shp_\e} \phi_0 \pd{\phi_{\e}}{\nu_\eps} \, d \mc{H}^{N-1}=
	\lim_{\e \to 0^+} \sqrt{\Lambda_\e} \int_{\partial \shp_\e} \phi_0 \, d \mc{H}^{N-1}= |\nabla \phi_0(x)|^2.
\end{equation}
This, combined with Theorem \ref{theo_sets_expansion}, yields Remark \ref{rmk:altra}.
\end{remark}

\appendix 
\section{Known properties of the optimal sets}\label{sec_app_prop_min}

In this appendix we collect some known results concerning problem \eqref{min_prob_eigen_ge} that we use in the paper.

The existence of a minimizer, in the weaker setting of quasi-open sets, is a consequence of a general 
result by Buttazzo and Dal Maso \cite{MR1217590}. The further regularity properties, both of the 
eigenfunctions and of the free boundaries, descend from  more recent papers  
\cite{BHP_lip_reg,BJ_min_open,BRT_reg_free_boundary_up_boundary} (see also the monography 
\cite{V_book_free_boun}).

We first recall the validity of a penalized version of \eqref{min_prob_eigen_ge} and state it with our notation. 
\begin{theorem}[{\cite[Theorem 2.9]{BHP_lip_reg}}]\label{theo_free_boundary_global}
Let $\phi_\e$ be as in \eqref{hp_phi_e}. Let
\begin{equation}\label{gamma_eps}
\gamma_\e :=\frac{2|\bx|\lambda_\e^2}{N|B_1(0)|^{2/N}\left(|\bx|-\e\right)^{2(N-1)/N}}.
\end{equation}
For every $\gamma>\gamma_\e$ we have
\begin{equation}\label{ineq_free_boundary_global}
\int_{\bx} |\nabla \phi_\e|^2 \, dx \le \int_{\bx} |\nabla v|^2 \, dx +\la_\e \left[1-\int_{\bx} v^2 \, dx\right]^++\gamma \left[|\{v\neq 0\}|-|\bx| +\e\right]^+,
\end{equation}
for every $v \in H^1_0(\bx)$.
\end{theorem}

The explicit value of the constant $\gamma_\e$ in \eqref{gamma_eps} is inferred from the proof of \cite[Theorem 2.9]{BHP_lip_reg}. The penalized formulation \eqref{ineq_free_boundary_global} readily provides a bound on the measure $|\Delta \phi_\e|$, that we now recall.
\begin{proposition}[{\cite[Lemma 4.2]{BHP_lip_reg}}]\label{prop_ineq_lap_etimates}
There exists $C_\e>0$,  depending only $N$ and on the constant $\gamma_\e$ defined in \eqref{gamma_eps}, such that
\begin{equation}\label{ineq_lap_estimates}
|\Delta\phi_\e|(B_r(x_0)) \le C_\e r^{N-1}
\end{equation} 
for every ball $B(x_0,r)$ such that $B(x_0,2r)\subset\bx$, 
where $|\Delta\phi_\e|$ is intend in the sense of Radon measures. 
\end{proposition}

The bound on the measure $|\Delta \phi_\e|$ is the basic tool to prove the Lipschitz continuity of $\phi_\e$, see \cite[Theorem 4.1]{BHP_lip_reg} for the local Lipschitz continuity. As we will need a global version, we refer here to the following result.

\begin{proposition}[{\cite[Proposition 5.18]{BRT_reg_free_boundary_up_boundary}}]\label{prop_equi_lip}
There exists $L>0$,  depending only on 
$N$ and $\bx$,  such that, for every $0\le\eps\le|\bx|/2$
\begin{equation}\label{ineq_grad_estimate}
\norm{\nabla \phi_\e}_{L^\infty(\bx)} \le  L.
\end{equation} 
\end{proposition}

We also recall that an optimality condition holds, in viscosity sense, on $\partial \shp_\e$.

\begin{proposition}[{\cite[Lemma 5.30]{BRT_reg_free_boundary_up_boundary}}] \label{prop_prob_visc}
For any $\e \in (0,\e_0)$, there exists $\Lambda_\e>0$ such that
\begin{equation}
\begin{cases}
-\Delta \phi_\e= \la_\e \phi_\e, \text{ in } \shp_\e,\\
|\nabla \phi_\e| = \sqrt{ \Lambda_\e}, \text{ in } \partial \shp_\e \cap \bx,\\
|\nabla \phi_\e| \ge  \sqrt{ \Lambda_\e}, \text{ in } \partial \shp_\e \cap \partial  \bx,
\end{cases}
\end{equation}
where the boundary conditions hold in viscosity sense.
\end{proposition}

The previous result allows the authors to apply the theory developed in \cite{CLS_boundary_improv_flatness,DS_flatness} and to obtain the full regularity of the free boundary $\partial \shp_\e$. We recall the definition of regular and singular point.
\begin{definition}\cite[Section 5.8-5.9]{BRT_reg_free_boundary_up_boundary}
We say that $x_0 \in \partial \shp_\e \cap \bx$ is  \textit{regular} if there exists a sequence $r_n \to 0^+$ as $n \to \infty$
such that
\begin{equation}
\lim_{n \to \infty}\frac{1}{r_n}u(x_0+r_n x)=\sqrt{\Lambda_\e} (x \cdot \nu)^+ \quad \text{ uniformly in any compact subset of $\R^N$},
\end{equation}
for some unitary vector $\nu \in \mb{S}^{N-1}$.
If $x$ is not regular than it is called  \textit{singular}. 
\end{definition}

\begin{theorem}[{\cite[Theorem 1.5]{BRT_reg_free_boundary_up_boundary},\cite[Theorem 1.2]{BJ_min_open}}]\label{theor_properties_Ke} \label{theo_reg_free}
Let $\e \in (0,\e_0)$, with $\e_0$ defined in \eqref{def_e0}. 
Let  $K_\e$, $\shp_\e$, and $\phi_\e$ be as in \eqref{K_e_def}, \eqref{D_e_def}, and \eqref{hp_phi_e}, \eqref{eq_phi_e} respectively.
Then
\begin{enumerate}[\rm (i)]
\item $\shp_\e= \{x \in \bx: \phi_\e(x)>0\}$ is connected,
\item $\shp_\e$ has locally finite perimeter  in $\bx$,
\item ${\rm Reg}(\partial \shp_\e \cap \bx)$ is an analytical hypersurface,
\item $\partial \shp_\e \cap \partial \bx \subset {\rm Reg}(\partial \shp_\e)$ and $ {\rm Reg}(\partial \shp_\e)$ is $C^{1,\frac{1}{2}}$ regular, 
\item
\begin{itemize}
\item  ${\rm Sing}(\partial \shp_\e) = \emptyset $ if $ N<N^*$,
\item  ${\rm Sing}(\partial \shp_\e)$ is a discrete set  if $ N=N^*$,
\item  $\dim_{\mc{H}}{\rm Sing}(\partial \shp_\e) < N-N^* $ if $ N>N^*$,
\end{itemize}
\end{enumerate}
where ${\rm Reg}$ and ${\rm Sing}$ are the regular and singular part of the free boundary respectively, see \cite[Definition 5.34]{BRT_reg_free_boundary_up_boundary} for further details, and $N^* \in \{5,6,7\}$ is the critical dimension, as defined in  \cite[Definition 1.5]{V_book_free_boun}.
\end{theorem}

We conclude this overview with a simple remark that is needed in Section \ref{sec_non_deg}.

\begin{remark} \label{remark_inner_ball_reg}
Let  $x_0 \in \partial \shp_\e \cap \bx$ and suppose that the interior ball condition holds at $x_0$, that is, there exists $y_0 \in \shp_\e$ and $R_0>0$ such that 
$B_{R_0}(y_0) \subset \shp_\e$ and $ x_0 \in  \overline{B_{R_0}(y_0)}$. Then, thanks to  \cite[Lemma 2.1, Remarks 2.2-2.4]{FV_dis_weights} and \cite[Remark 5.28, Lemma 5.31]{BRT_reg_free_boundary_up_boundary}, $x_0$ is a regular point.
\end{remark}

\bigskip

\textbf{Acknowledgments.} Work partially supported by: PRIN-20227HX33Z ``Pattern formation in nonlinear 
phenomena'' - funded by the European Union-Next Generation EU, Miss. 4-Comp. 1-CUP D53D23005690006,  
and by the MUR-PRIN project no. 2022R537CS ``NO$^3$'' granted by the European Union -- Next Generation EU.
The authors are members of the INdAM-GNAMPA group.

\bigskip

\textbf{Data Availability.} Data sharing not applicable to this article as no datasets were generated or analyzed during the current study.

\bigskip

\textbf{Disclosure statement.} The authors report there are no competing interests to declare.

\bibliography{Minimal_eigen_Laplacian_bibliografy}	

\bibliographystyle{abbrv}

\medskip
\small
\begin{flushright}
\noindent 
\verb"benedetta.noris@polimi.it"\\
\verb"gianmaria.verzini@polimi.it"\\
Dipartimento di Matematica, Politecnico di Milano\\ 
piazza Leonardo da Vinci 32, 20133 Milano (Italy)
\bigskip

\verb"giovanni.siclari@sns.it"\\
Centro di Ricerca Matematica Ennio De Giorgi\\
Classe di Scienze, Scuola Normale Superiore\\
piazza dei Cavalieri 7, 56126 Pisa (Italy)\\
\end{flushright}

\end{document}